\documentclass[12pt]{article}
\usepackage{latexsym,amsmath,amssymb,epsf,amsthm,times,amscd}
\usepackage[latin1]{inputenc}
\usepackage[usenames]{color}
\usepackage{epsf}
\usepackage{graphicx}
\usepackage{graphics}
\usepackage{subcaption}

\newfont{\bb}{msbm10 at 12pt}
\newfont{\bbp}{msbm10 at 9pt}

\def\r{\hbox{\bb R}}

\def\d{\hbox{\bb D}}

\def\b{\hbox{\bb B}}

\def\s{\hbox{\bb S}}
\def\sp{\hbox{\bbp S}}

\newcommand{\U}{\mathcal{U}}

\newcommand{\abs}[1]{\left\vert #1 \right\vert}
\newcommand{\set}[1]{\left\{#1\right\}}
\newcommand{\meta}[2]{\langle #1,#2 \rangle }

\newcommand{\pscalar}[2]{\left\langle #1, #2 \right\rangle}

\numberwithin{equation} {section}

\begin{document}

\theoremstyle{plain}\newtheorem{lemma}{Lemma}[section]
\theoremstyle{plain}\newtheorem{definition}{Definition}[section]
\theoremstyle{plain}\newtheorem{theorem}{Theorem}[section]
\theoremstyle{plain}\newtheorem{proposition}{Proposition}[section]
\theoremstyle{plain}\newtheorem{remark}{Remark}[section]
\theoremstyle{plain}\newtheorem{corollary}{Corollary}[section]

\begin{center}
\rule{15cm}{1.5pt} \vspace{.6cm}

{\Large \bf An overdetermined eigenvalue problem \\[3mm]and the Critical Catenoid conjecture} \vspace{0.4cm}

\author[Jos\'{e} M. Espinar$\mbox{}^\dag$ and Diego A. Mar\'{i}n$\mbox{}^\ddag$

\vspace{0.3cm} \rule{15cm}{1.5pt}
\end{center}

\vspace{.3cm}

\noindent $\mbox{}^\dag$ Department of Geometry and Topology and Institute of Mathematics (IMAG), University of Granada, 18071, Granada, Spain; e-mail:
jespinar@ugr.es\vspace{0.2cm}

\noindent $\mbox{}^\ddag$ Department of Geometry and Topology and Institute of Mathematics (IMAG), University of Granada, 18071, Granada, Spain; e-mail: damarin@ugr.es

\vspace{.3cm}

\begin{abstract}
We consider the eigenvalue problem $\Delta^{\sp^2} \xi + 2 \xi=0 $ in $ \Omega $  and $\xi = 0 $  along  $ \partial \Omega $, being $\Omega$ the complement of a disjoint and finite union of smooth and bounded simply connected regions in the two-sphere $\mathbb{S}^2$. Imposing that $\abs{\nabla \xi}$ is locally constant along $\partial \Omega$ and that $\xi$ has infinitely many maximum points, we are able to classify positive solutions as the rotationally symmetric ones. As a consequence, we obtain a characterization of the critical catenoid as the only embedded free boundary minimal annulus in the unit ball whose support function has infinitely many critical points.
\end{abstract}



\section{Introduction}
Consider a Riemannian manifold $(M,g)$ and let $\Omega \subset M$ be a bounded domain. Additionally, let $f$ be a positive real-valued Lipschitz function. Consider the following elliptic equation:

\begin{equation}\label{DP}
	\left\{ \begin{matrix}
		\Delta \xi + f(\xi)=0 & \text{ in } & \Omega ,\\[2mm]
		\xi >0 & \text{ in } & \Omega, \\[2mm]
		\xi = 0 & \text{ along } & \partial \Omega.
	\end{matrix}\right.
\end{equation}

This problem is commonly known as a Dirichlet problem, and it typically has a unique solution when one exists. Moreover, we introduce the following additional condition to the problem:
\begin{equation}\label{NC}
	\frac{\partial \xi}{\partial\nu} = c \quad \text{along} \quad \partial \Omega , \quad c \in \mathbb{R}^+,
\end{equation}where $\nu$ denotes the inner normal vector to $\partial \Omega$, it transforms into an overdetermined elliptic problem (OEP) and solutions only exist for specific domains $\Omega$. Consequently, a solution to the combined problem \eqref{DP} and \eqref{NC} is represented as a pair $(\Omega, \xi)$, emphasizing that not only the function $\xi$ but also the domain $\Omega$ are essential for solving overdetermined problems.

For instance, if we consider $\Omega$ as a geodesic ball in a space form and $\xi$ as a function dependent solely on the distance from the center of the ball that satisfies \eqref{DP}, it becomes evident that $\xi$ satisfies condition \eqref{NC}. This naturally raises the question of whether these are the only types of solutions to the OEP on bounded domains. The interest in such problems traces back to Serrin's seminal paper \cite{Se}, where he established the rigidity of problem \eqref{DP} and \eqref{NC} when $\Omega$ is a bounded domain in Euclidean space and $f\equiv 1$. Serrin's work showed that if $(\Omega, \xi)$ is a $\mathcal{C}^2$ solution (i.e., both the solution $\xi$ and the domain $\Omega$ exhibit $\mathcal{C}^2$ regularity) to the OEP \eqref{DP} and \eqref{NC}, then $\Omega$ must be a round ball and $\xi$ must be rotationally symmetric. Serrin employed the Alexandrov reflection method and a modified maximum principle for domains with corners to establish this result. Subsequently, Pucci and Serrin in \cite{Pu} extended this result to include a general $f$ with Lipschitz regularity. Serrin's paper marked the inception of OEPs in which the overdetermined condition \eqref{NC} implies the symmetry of the domain $\Omega$. 

However, beyond these well-known examples, there exist more solutions to OEPs, apart from the previously mentioned trivial ones. The first non-symmetric solutions to an OEP in the Euclidean space were constructed by P. Sicbaldi in \cite{Sic}. He showed the existence of a family of solutions to \eqref{DP} and \eqref{NC} in $\mathbb{R}^n$, $n \geq 3$, and $f(t)= \lambda t$ (a linear function). These solutions were perturbations of a straight cylinder $\mathbb{B}^{n-1} \times \mathbb{R}$ obtained by reformulating the overdetermined problem in terms of a differential operator and applying the Crandall-Rabinowicz bifurcation theorem. Later, many non-trivial solutions have been discovered for various OEPs using bifurcation of symmetric domains in more general Riemannian manifolds (cf. \cite{DEP,FMW,Pa,Ruiz,Sch} and references therein).

Serrin's seminal work \cite{Se} holds significance, not only for the classification result itself, but for introducing the `moving plane method" in the study of partial differential equations (PDEs) and, since then, this method has been applied to prove similar results in other space forms, the closed hemisphere $\textup{cl}(\mathbb{S}^n_+)$ and hyperbolic space $\mathbb{H}^n$, or extended to unbounded domains (cf. \cite{BCN,Ber,EFM,EMao,EM,KP,Mir,Rei2} and references therein).

Another widely-used method in the study of OEPs is the one introduced by Weingberger \cite{We}. Weingberger employed the maximum principle with a specific function (a ``P-function") and a Pohozaev identity to establish Serrin's classification in an elementary manner. This method, now known as the ``P-function method", has been successfully adapted to different equations. For example, it has been employed to prove rigidity statements in space forms (cf. \cite{CV,QXia}), and product manifolds (cf. \cite{Far}).

While various analytical methods have been applied to study overdetermined problems (e.g., Morse theory \cite{Pace,Pace2} and complex analysis \cite{Kav,EM}), in recent years methods derived from differential geometry, particularly those related to the theory of minimal and constant mean curvature (CMC) hypersurfaces, have been successfully applied to OEPs in space forms (cf. \cite{DPW,EM,Ros}). Indeed, Serrin's theorem in \cite{Se} can be seen as the OEP version of Alexandrov's well-known theorem regarding the uniqueness of compact embedded CMC hypersurfaces in Euclidean space (cf. \cite{Ale}).

An interesting connection between OEPs and the theory of minimal surfaces can be found in the paper by H\'{e}lein et al. \cite{Hel}, where they established a link between a certain type of minimal surfaces and solutions to the problem \eqref{DP} and \eqref{NC} in $\Omega \subset \mathbb{R}^n$ with $f \equiv 0$, the ``one phase problem". In particular, they constructed a Weierstrass-like representation of solutions to the previous problem, which Traizet subsequently used in \cite{Trai} to classify solutions to the problem. 

Another connection, with greater relevance to the objective of this paper in which the theory of minimal surfaces is employed to classify solutions to an OEP, is presented in R. Souam's paper \cite{Sou} where he investigates the overdetermined eigenvalue problem:

\begin{equation}\label{eqSouam}
	\left\{ \begin{matrix}
		\Delta^{\sp^2} \xi + 2 \xi=0 & \text{ in } & \Omega,\\[2mm]
		\xi = const. & \text{ along } & \partial \Omega, \\[2mm]
		\frac{\partial \xi}{\partial\nu} = const. & \text{ along } & \partial \Omega,
	\end{matrix}\right.
\end{equation}being $\Omega$ a $\mathcal{C}^2$ domain in $(\mathbb{S}^2, g_{\sp^2})$, the unit sphere with the round metric. It is a known fact (cf. \cite{Har,Re}) that if $\xi$ solves the first equation of the previous system, then the map

\begin{equation}\label{paramSouam}
	X(z) := \nabla^{\sp^2} \xi (z) + \xi (z) \cdot z, \quad z \in \Omega
\end{equation}produces a branched minimal surface. Conversely, the support function of any branched minimal surface parameterized by its Gauss map provides a solution to the equation $\Delta^{\sp^2} \xi + 2 \xi = 0$. Using this correspondence and Nitsche Theorem \cite{JNit}, Souam established that if $(\Omega, \xi)$ is a solution to \eqref{eqSouam} and $\Omega$ is simply connected, then $\Omega$ must be a geodesic disk and $\xi$ must exhibit rotational symmetry.

Recall that, as said above, the moving plane method in the sphere can only be applied for domains contained in a closed hemisphere $\textup{cl}(\mathbb{S}^n_+)$. Recently, Espinar-Mazet \cite{EM} extended the classification of simply connected domains (not necessarily contained in any closed hemisphere) in $\s ^2$ supporting an overdetermined solution to $\Delta^{\sp^2} \xi + f( \xi )=0$ as the rotationally symmetric ones for a wider class of functions $f$, in particular, this method classify positive solutions to \eqref{eqSouam} in simply connected domains. In fact, Serrin's result is not fully generalizable in the sphere $\s^2$, for there exist non-symmetric solutions to Serrin's overdetermined problem with equation $-\epsilon \Delta^{\sp^2} \xi+ \xi-\xi^{p}$ ($\epsilon>0,p>1$) defined in simply connected domains (c.f. \cite{Ruiz2}).

This paper focuses on positive solutions to the eigenvalue problem:
\begin{equation}\label{DPS2}
	\left\{	\begin{matrix}
		\Delta^{\sp^2} \xi + 2 \xi=0 & \text{ in } & \Omega ,\\[2mm]
		\xi = 0 & \text{ along } & \partial \Omega.
	\end{matrix}\right.
\end{equation}
Here, we assume that $\Omega \subset \s^2$ is a $\mathcal{C}^2$-domain and that $\xi$ is of class $\mathcal{C}^2$ up to the boundary, so it follows that $\partial \Omega$ is analytic (and so is the function up to the boundary) because of the regularity results of \cite{Nir}.

Since overdetermined positive solutions to \eqref{DPS2} on simply connected domains are rotationally symmetric (cf. \cite{EM,Sou}), we consider ``finite type domains", which are the complements of a finite number of disjoint simply connected domains within $\mathbb{S}^2$. The concept at hand is to employ the correlation established in \cite{Sou}, albeit in a distinct fashion, and the methodologies recently advanced in \cite{ABM} for the classification of solutions to the Serrin equation within annular domains in $\mathbb{R}^2$. Such techniques (cf. also \cite{ABM,Borg,Borg2,Chr}) are particularly interesting for solutions to OEPs in domains in the sphere where the moving plane method is not fully available.

We aim to classify rotationally symmetric solutions to \eqref{DPS2} under certain conditions, which, in turn, leads to the classification of minimal surfaces with specific properties (see Section 6). In particular, we focus on solutions to \eqref{DPS2} that exhibit infinitely many maximum points. Our primary classification result is as follows:

\begin{quote}\label{theor_A}
	\textbf{Theorem A}: Let $(\Omega, \xi)$ be a positive solution to \eqref{DPS2}, with $\Omega$ being a topological annulus with $\mathcal{C}^2$ boundary. Suppose that $\xi$ has infinitely many maximum points and that the norm of its gradient is locally constant along the boundary, i.e.,
	
	\begin{equation}\label{NC_modified}
		|\nabla^{\sp^2} \xi| = b_{\Gamma} > 0, \quad \Gamma \in \pi_0 (\partial \Omega),
	\end{equation}where $b_\Gamma$ is a constant for each connected component $\Gamma$ of $\partial \Omega$. Then $\Omega$ is a rotationally symmetric neighborhood of an equator, and $\xi$ exhibits rotational symmetry with respect to the axis perpendicular to the plane defining this equator.
\end{quote}

\begin{remark}
It is worth noting that the constants considered in \eqref{NC_modified} could be distinct for different connected components. 
\end{remark}

Using the correspondence established in Souam's paper, we can derive an intriguing consequence of Theorem A. Let $\Sigma \subset \mathbb{R}^3$ be an open embedded minimal annulus with boundary, each boundary component intersecting orthogonally a sphere centered at the origin  (possibly of different radius). Such surfaces are referred to as ``minimal surfaces with free boundaries" (see Definition \ref{surfacesWithFreeBoundaries}). It is evident that if $\xi$ is a solution to \eqref{eqSouam}, and we consider the parametrization \eqref{paramSouam}, we obtain a surface of this type (a priori, not necessarily embedded). We arrive at the following result:

\begin{quote}\label{theor_B}
	\textbf{Theorem B}: Let $\Sigma \subset \mathbb{B}^3$ be an embedded minimal annulus with free boundaries, and suppose its support function (or distance function) has infinitely many critical points. Then $\Sigma$ is a piece of a rotationally symmetric catenoid.
\end{quote}
\begin{remark}
It is not hard to see that, in the above situation, the set of critical points of the support function coincides with the set of critical points of the distance (to the origin) function.
\end{remark}

Notably, if $\Sigma \subset \mathbb{B}^3$ (where $\mathbb{B}^3$ is the Euclidean unit ball centered at the origin) and $\partial \Sigma \subset \mathbb{S}^2$, then $\Sigma$ is a ``free boundary surface" within the unit ball. Such surfaces have been extensively studied in recent years (see \cite{Car} or \cite{Li} for a survey of recent results). It is important to note that there exists a unique embedded catenoid that intersects the unit sphere $\mathbb{S}^2$ orthogonally, known as the ``critical catenoid" (see \cite{Dev}). Consequently, an immediate consequence of the aforementioned result is:

\begin{quote}\label{corollary_C}
	\textbf{Corollary C}: Let $\Sigma \subset \mathbb{B}^3$ be an embedded free boundary minimal annulus, and suppose its support function has infinitely many critical points. Then $\Sigma$ is the critical catenoid.
\end{quote}

\begin{remark}
To our knowledge, this is the first classification result of the critical catenoid that does not use Schoen-Fraser classification result by first Steklov eigenfunctions (c.f. \cite[Theorem 1.2]{FS2}). 
\end{remark}

The question of whether the critical catenoid is the only embedded free boundary minimal annulus within the unit ball remains open. In recent years, this question has garnered significant attention within the mathematical community. The prevailing belief is that the answer is affirmative, and numerous partial results have been obtained to substantiate this belief (see, e.g., \cite{Dev,FS,Kus,McG}). This corollary can be regarded as a small contribution to the proof of the so-called ``critical catenoid conjecture," an outstanding open problem in the field.

\begin{remark}
The tools employed in this paper have the potential for extension to positive solutions of the general eigenvalue problem, specifically, positive solutions to $\Delta \xi + \lambda \xi = 0$ within the domain $\Omega$ and incorporating overdetermined boundary conditions. However, our primary focus lies in exploring the geometric applications associated with the particular case of $\lambda = 2$.
\end{remark}

\subsection{Organization of the paper}
\mbox{}
Our primary focus in Section \ref{sec_2} is to compute the rotationally symmetric solutions to \eqref{DPS2}. We present these solutions as a one-parameter family of normalized functions $(\Omega_R, \xi_R)$ that depends on the height relative to the horizontal equator of $\mathbb{S}^2$, $R \in [0,1]$. We also interpret this family of solutions as the support functions of a family of vertical catenoids. 

In Section \ref{sec_3}, we narrow our focus to a region $\mathcal{U}$ within $\Omega$ without maximum points for a particular positive solution $(\Omega, \xi)$ to \eqref{DPS2}. We introduce a parameter $R(\U)$ in a similar manner to \cite{ABM}. We then establish that if $\Omega$ is not a topological disk, then $R(\U) = \bar R$ is well-defined and takes values in $[0,1)$ (Theorem \ref{theor_CriticalHeight}). This allows us to associate a model solution $(\Omega_{\bar R}, \xi_{\bar R})$ with a general solution within the region $\U$. 

Section \ref{sec_4} introduces what is referred to as a ``pseudo-radial function" (see Definition \ref{psudoRadialFunctions}), which we use to derive estimates for the norm of the gradient of $\xi$ along its level sets (see Theorem \ref{theor_Gradient}) and for the geodesic curvature of the level sets (see Theorem \ref{curvature_Theorem}) within $\U$. We also provide a relation between the length of the zero level set and the top level set (Proposition \ref{lengthBounds}).

In Section \ref{sec_5}, we examine positive solutions to the OEP \eqref{DPS2} and \eqref{NC_modified}. We establish an estimate for the length of the zero level sets of the region $\partial \U$ using the overdetermined condition (Proposition \ref{lengthZeroLevelSets}). This estimate, combined with those obtained in Section 4, enables us to conclude that either $\U \subset \Omega$ or $\Omega \setminus \U$ is contained within an open hemisphere. The statement of Theorem A follows from the moving plane method. In fact, we are able to establish a more general result, as indicated in Theorem \ref{rigidityConstantNeumannOEP}.

In Section \ref{sec_6}, we delve into the geometric consequences of the classification results obtained in Section \ref{sec_5}. Through the study of minimal surfaces with free boundaries, we ascertain that an embedded minimal annulus with each boundary component intersecting a sphere (centered at the origin) orthogonally, must possess an injective Gauss map and its support function must have a definite sign. This immediately leads to the statement of Theorem B and Corollary C using the correspondence from Souam's paper.


\section{The model solutions}\label{sec_2}

In this section, we will describe the family of positive and bounded rotationally symmetric solutions $(\Omega_{R}, \xi_{R})$ to the Dirichlet problem \eqref{DPS2}, depending on a parameter $R \in [0,1]$. In this family, $\Omega_0$ is a symmetric tubular neighborhood of an equator and $\Omega_1$ is an open hemisphere. 

Let $\r ^3 $ denote the usual Euclidean space, where $(x,y,z)$ represent cartesian coordinates and $\pscalar{\cdot}{\cdot}$ is the Euclidean scalar product. We parametrize the unit sphere in cylindrical coordinates:
$$ \s ^2 = \set{(\sqrt{1-r^2} \cos \theta , \sqrt{1-r^2} \sin \theta , r) \, : \, r \in [-1,1] , \, \theta \in [0,2\pi)} ,$$hence the induced metric in $\s^2$ in these coordinates is given by
$$ g_{\sp^2} = \frac{1}{1-r^2} dr^2 + (1-r^2) d\theta ^2 .$$

Define the orthonormal frame in $\s^2 \setminus \set{{\bf s}, {\bf n}}$, where ${\bf s} := (0,0,-1)$ and ${\bf n} := (0,0,1)$ denote the south and north pole respectively, given by 
$$ n = \sqrt{1-r ^2} \partial _r \quad \text{and} \quad t = \frac{1}{\sqrt{1-r^2}} \partial _\theta .$$
Then, the Christoffel symbols associated to this frame can be computed from
$$  \nabla _{n} n =0 , \qquad  \nabla_{n} t =0 , \qquad    \nabla_{t} n =  - \frac{r}{\sqrt{1-r^2}} t , \qquad  \nabla_{t} t = \frac{r}{\sqrt{1-r^2}}n .$$

From now on, as we will be always considering differential equations on the sphere, we denote by $\nabla$, $\Delta$ and $\nabla^2$ the gradient, laplacian and hessian operators in $\mathbb{S}^2$.

\subsection{Rotationally symmetric solutions}

From the first equation of \eqref{DPS2} using the cylindrical coordinates, assuming that $\xi$ doesn't depend on $\theta$, we obtain
\[
(1-r^2) \frac{\partial^2 \xi}{\partial r^2} - 2 r \frac{\partial \xi}{\partial r} + 2 \xi =0.
\]
This is an ordinary differential equation of second order, and it has a two-parameter family of solutions:
\begin{equation}\label{OdeSol}
\xi_{\alpha, \beta} (r)= \alpha (1-r \, \textup{arctanh}(r))+\beta r, \quad r \in (-1,1), \quad \alpha, \beta \in \mathbb{R}.
\end{equation}

Hence, a function $\xi _{\alpha , \beta }$ will be a solution to \eqref{DPS2} if there exist $a, b \in [-1,1]$, $a<b$, depending on $\alpha $ and $\beta$ such that 
\[
\xi _{\alpha , \beta}(a)=\xi(b)=0 \quad \text{and} \quad  0 < \left(\xi _{\alpha , \beta}\right)_{| (a,b)} < + \infty .
\]

\begin{itemize}
\item {\bf Case 1: $\alpha =0$ and $\beta \neq 0$:}  We obtain the \textit{reescale height function} $\xi_{0,\beta} (r)= \beta r$. This is a monotone function that takes the zero value only at the origin, and we obtain the solution $(\Omega _{1},\xi_{0, \beta})$, where $\Omega _1$ is the open hemisphere centered either at the north pole ${\bf n}$ if $\beta>0$ or at the south pole ${\bf s}$ if $\beta<0$.
\item  {\bf Case 2: $\alpha  \neq 0$ and $\beta =0 $:} We obtain the function
\[
\xi_{\alpha,0}(r)= \alpha (1-r \, \text{arctanh}(r)), \quad r \in (-1,1).
\]
If we assume that $\alpha >0$, then $\xi_{\alpha,0}(r)$ has a positive maximum at the origin since $\xi_{\alpha,0}(-r) = \xi_{\alpha,0}(r)$, it is monotonic in both intervals $(-1,0)$ and $(0,1)$ and 
 \[
 \lim_{r \to 1^{-}} \xi_{\alpha,0}(r) = \lim_{r \to -1^{+}} \xi_{\alpha,0}= - \infty .
 \]
Hence, it is clear that there exists $\bar{r} \in (0,1)$ such that $\xi_{\alpha,0}(\bar{r})=\xi_{\alpha,0}(-\bar{r})=0$ and $\left(\xi_{\alpha,0}\right)_{| (-\bar{r}, \bar{r})}>0$. In this case, we can easily compute that 
$$ \frac{\partial \xi _{\alpha ,0}}{\partial r} (0) = 0  \text{ and } (\xi_{\alpha,0})_{{\rm max}} = \xi _{\alpha , 0} (0)= \alpha .$$

Therefore, $(\Omega _0 , \xi _{\alpha , 0} ) $ is solution to \eqref{DPS2} such that $\Omega _0 $ is a symmetric tubular neighborhood of the equator $\s ^2 \cap \set{z =0}$. Finally, observe that $\xi_{-\alpha,0} (r) = - \xi _{\alpha , 0} (r)$ which means that it is negative in $(-\bar{r},\bar{r})$ and positive in $(-1,1) \setminus [-\bar{r}, \bar{r}]$. However, 
\[
 \lim_{r \to 1^{-}} \xi_{-\alpha,0}(r) = \lim_{r \to -1^{+}} \xi_{-\alpha,0}= + \infty ,
 \]
that is, $\xi_{-\alpha,0}$ is not bounded. 
\item  {\bf Case 3: $\alpha  \neq 0$ and $\beta \neq 0 $:}  Consider $\alpha >0$. Since 
$$ \xi _{\alpha , \beta } (r)  = \alpha  \, \xi _{1 , \omega } (r) \text{ and }\xi_{1,\omega }(-r)=\xi_{1,-\omega}(r) $$
where $\omega =\beta/\alpha \neq 0$, we can write solutions to \eqref{OdeSol} in a more interesting way
\[
\xi_{\alpha,  \beta } (r)=  \alpha \, \xi _{1, \omega } (r) , \text{ where } \xi _{1, \omega } (r) = 1-r \, \textup{arctanh}(r)+ \omega r  .
\]
We first compute the zeros of $\xi_{\alpha, \beta}$. Since $\xi_{1, \omega}(0)=1$ and  $\lim_{r \to 1^{-}} \xi_{1, \omega }(r) = \lim_{r \to -1^{+}} \xi_{1,\omega}= - \infty $, then $\xi _{1, \omega}$ there exist uniques $ r_- \in (-1,0)$ and $ r_+ \in (0,1)$ such that $\xi_{1, \omega}(r_{-})=\xi_{1, \omega}(r_{+})=0$. Moreover, 
$$ - \frac{\partial \xi_{1, \omega }}{\partial r}(r) =\text{arctanh}(r)+\frac{r}{1-r^2}-\omega , \quad r \in (r_{-},r_{+})  $$being 
$$ \frac{\partial \xi_{1, \omega }}{\partial r}( r_-) = -\frac{1}{ r_- (1-  r_- ^2)} >0 \text{ and } \frac{\partial \xi_{1, \omega }}{\partial r}( r_+) = -\frac{1}{ r_+ (1-  r_+ ^2)} < 0,$$
so we conclude that there exists a unique $R \in ( r_{-},  r_{+})$ such that $\frac{\partial \xi_{1,\omega}}{\partial r}(R)=0$, which is a maximum. We have
\begin{equation}\label{rparameter}
	 \omega =\text{arctanh}(R)+\frac{R}{1-R^2},
\end{equation}
and the value of the maximum of $\xi_{\alpha, \beta}$ is given by
\[
(\xi_{\alpha, \beta })_{\textup{max}}= \alpha \, \xi_{1, \omega } (R)=\frac{\alpha}{1-R^2}.
\]

Finally, we note that if $\alpha<0$ then $\xi_{\alpha, \beta }$ has a negative sign in $( r_{-},  r_{+})$ and a positive sign in $(-1,1)\setminus [ r_{-}, r_{+}]$; but $\xi _{\alpha , \beta} $ is not bounded in $(-1,1)\setminus [ r_{-},  r_{+}]$.

\end{itemize}

Summarizing, we have obtained a two-parameter family of solutions to the Dirichlet problem \eqref{DPS2}, where the parameters are a scale parameter $\alpha$ and a height parameter $R$, where $R$ indicates the parallel in $\s ^2$ where the curve of maximum points is located. In fact, we can describe the solutions previously obtained as
\[
\xi_{\alpha,R}(r)=\alpha\left(1-r\textup{arctanh}(r)+\left(\textup{arctanh}(R)+\frac{R}{1-R^2}\right)r\right) , \quad \forall r \in [ r_{-}(R), r_{+}(R)].
\]
However, note that the eigenvalue equation of the Laplace-Beltrami operator on a general Riemannian manifold is invariant under a change of scale. Thus, the parameter $\alpha$ is not important and for each $R$ we can fix a value of the scale to obtain a 1-parameter family of rotationally symmetric solutions to \eqref{DPS2}. Although the most natural way of obtaining the 1-parameter family is to fix $\alpha=1$, we will use a different normalization for geometrical reasons. 
\begin{figure}[htb]
	\centering
	\includegraphics[width=12cm, height=7cm]{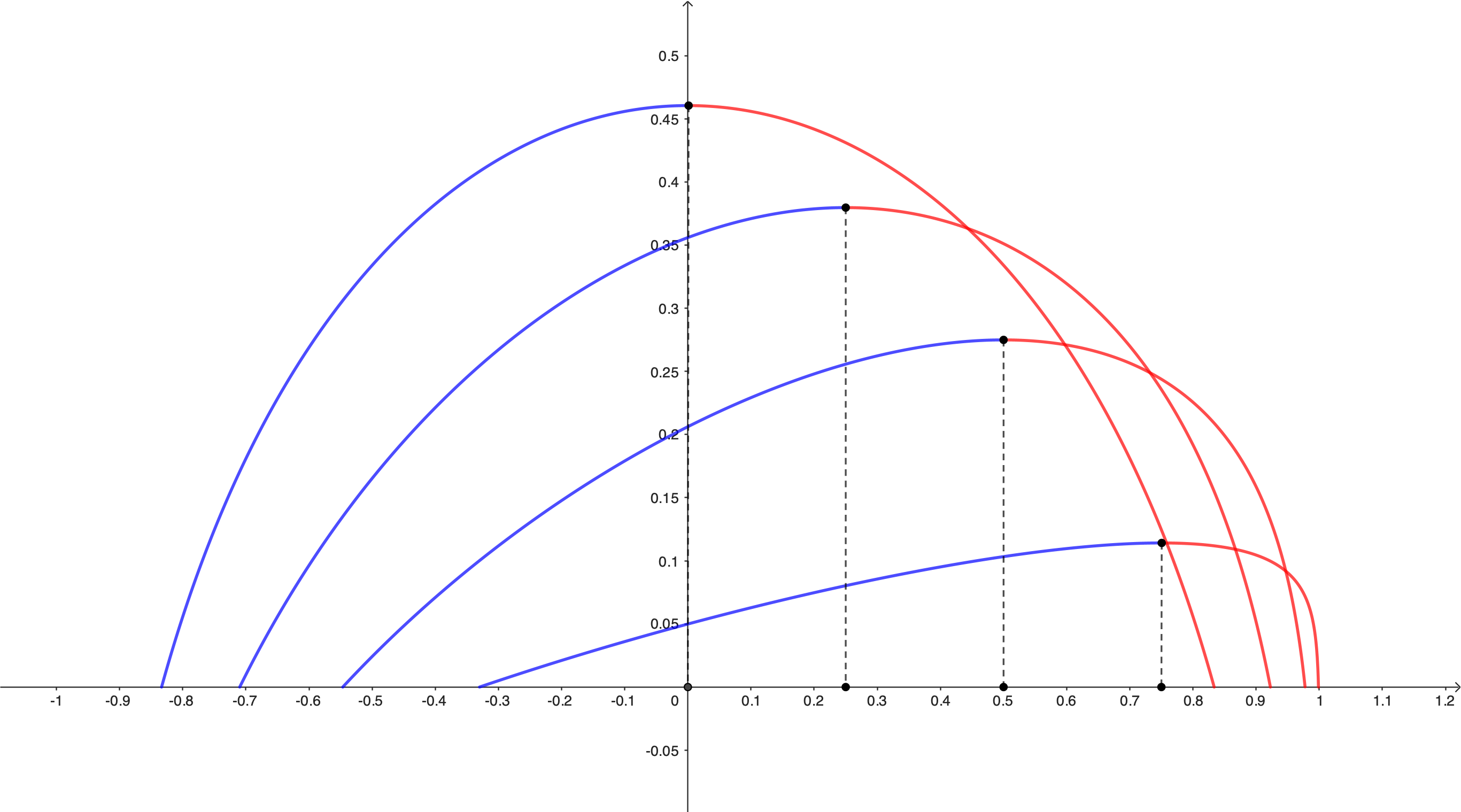}
	\caption{The diagram shows the graphs of the functions that define the model solutions $\xi_R$ with $R \in \{0,0.25,0.5,0.75\}$. In the blue part, the function is increasing, while in the red part, the function is decreasing, and the function takes its maximum values at the point $x=R$. With the normalization at hand, the value of $(\xi_{R})_{\textup{max}}$ goes to zero as $R$ approaches $1$, so $\xi _R \to 0$ in compact sets away from the north pole. In order to embrace the positive solution in the disk-type domain $\Omega _1$ we have to lose continuity at $R=1$ in the family $\xi _R$ (cf. Proposition \ref{prop_2.1}).}
\end{figure}
\begin{proposition}\label{prop_2.1}
Let $(\Omega,\xi)$ be a positive bounded rotationally symmetric solution to \eqref{DPS2}. Then, up to a rotation, a reflection with respect to the plane $\{z=0\}$ and a dilation,  $(\Omega, \xi)=(\Omega_R, \xi_{R})$ for $R \in [0,1]$ as described below:
\begin{itemize}
\item If $R=1$, $(\Omega_1,\xi_{1})$ is a solution to \eqref{DPS2} in a disk-type domain, where $\Omega_1$ is the open hemisphere centered at the north pole ${\bf n}$ and $\xi_{1}$ is the height function  $\xi _{1} (r) = r$.

\item If $R \in [0,1)$, then $(\Omega _R , \xi _{R})$ is a solution to \eqref{DPS2} in a annular-type domain given by
	\[
	\xi_{R}(r)= \alpha (R) \left(1-r \, \textup{arctanh}(r)+\left(\textup{arctanh}(R)+\frac{R}{1-R^2}\right)r\right) , \quad \forall r \in [ r_{-}(R), r_{+}(R)],
	\]
	where $-1< r_{-} (R)<0<  r_{+} (R)<1$ with $-r_{-} (0)= r_{+} (0)=\bar r$, being $\alpha(R)=r_+(R) \sqrt{1-r_+ (R)^2}$ and 
	\[
	\Omega_{R}=\left\{ \left(\sqrt{1-r^2}\cos(\theta),\sqrt{1-r^2}\sin(\theta),r\right) \in \s^2 ~\colon~  r \in [ r_{-}(R),  r_{+}(R)],\,\theta \in [0, 2\pi)  \right\}. 
	\]
Moreover, 
$$ \max_{\textup{cl}(\Omega)}\xi_R:= (\xi_{R })_{\textup{max}}=  \xi_{R } (R)=\frac{\alpha (R)}{1-R^2}=\frac{r_+(R) \sqrt{1-r_+ (R)^2}}{1-R^2} $$and we denote the boundary components as
\begin{equation}\label{boundaryComponents}
	\Gamma_{\pm}(R)=\left\{ \left(\sqrt{1- r_{\pm}(R)^2}\cos(\theta),\sqrt{1- r_{\pm}(R)^2}\sin(\theta), r_{\pm}(R)\right) \in \s^2 ~\colon~  \theta \in [0,2\pi)  \right\},
\end{equation}
where $\Gamma_{+}(R)$ and $\Gamma_{-}(R)$ are called the \textit{upper} and \textit{lower} component of $\partial \Omega _R$ respectively (with respect to the equator $\s^2 \cap \{z=0\}$).	
\end{itemize}
\end{proposition}

Observe that the constants $ r_-(R)$ and $ r_+ (R)$ are smooth functions of $R$ since they are implicit solutions to the equation
\begin{equation}\label{implicitH}
1-r \, \textup{arctanh}(r)+\left(\textup{arctanh}(R)+\frac{R}{1-R^2}\right)r = 0.
\end{equation}

\begin{figure}[htb]
	\centering
	\includegraphics[width=10cm, height=7cm]{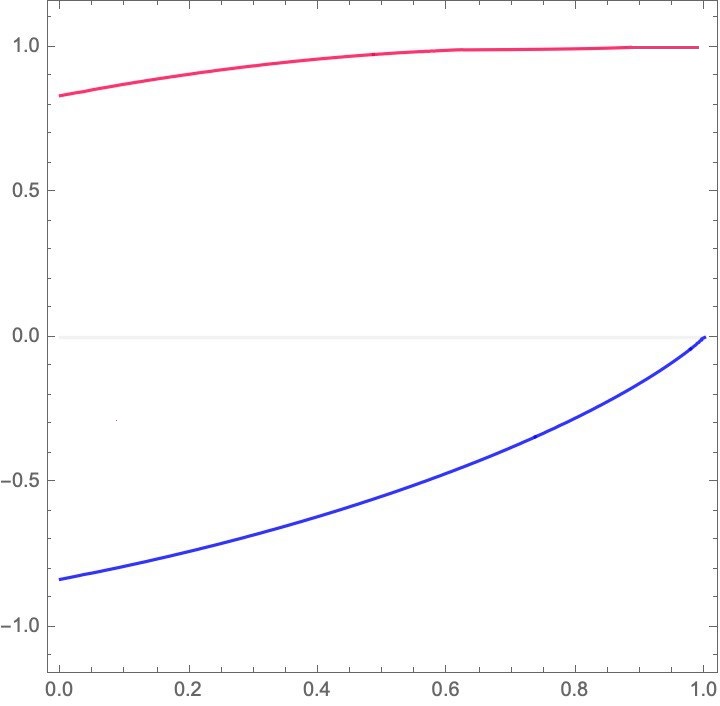}
	\caption{The graphs of $r_+$ (above) and $r_-$ (below) as functions of the parameter $R$. One can see that $-r_-(0)=r_+ (0)$ and $r_+ \to 1$ and $r_- \to 0$ as $R$ goes to $1$.} 
\end{figure}

Taking into account that $r_{\pm}(R) \neq 0$ for some $R \in [0,1]$ by \eqref{implicitH}, we have that
\begin{equation}\label{derivH}
\frac{\partial  r_{\pm}}{\partial R} (R)= \frac{2  r_{\pm}^2 (1-  r_{\pm}^2)}{(1-R^2)^2}, \quad \forall R \in [0,1),
\end{equation}
so both $ r_{-}$ and $ r_{+}$ are increasing functions of $R$. The following properties of our family of solutions are straightforward.

\begin{proposition}\label{propertiesXi}
Let $(\Omega_R, \xi_R)$, $R \in [0,1)$, be one of the model solutions described in Proposition \ref{prop_2.1}. Set $\alpha (R)=\alpha$. Then we have the following identities:
\begin{itemize}
	\item  The derivatives of the function up to second order:
	\begin{equation*}
	\frac{\partial \xi_{R}}{\partial r} (r)= \frac{1}{r} \left(\xi_{R}-\frac{\alpha}{1-r^2}\right) \quad \textup{and} \quad \frac{\partial^2 \xi_{R}}{\partial r^2} (r)=-\frac{2 \alpha}{(1-r^2)^2}.
	\end{equation*}
\item The gradient and Hessian :
\begin{equation*}
	\begin{split}
	&\nabla \xi_{R} (r)= \frac{\sqrt{1-r^2}}{r}  \left(\xi_{R}-\frac{\alpha}{1-r^2}\right) n \quad \textup{and}  \\
	 &\nabla^2 \xi_{R} (r)= \begin{pmatrix}
		-\xi_{R}-\frac{\alpha}{1-r^2} & 0 \\
		0 & -\xi_{R}+\frac{\alpha}{1-r^2}
	\end{pmatrix}.
	\end{split}
\end{equation*}
\item The norm of the gradient along the zero level sets:
\begin{equation*}
	\abs{\nabla \xi_R}=1 \, \textup{along }\Gamma_{+}(R) \quad \textup{and} \quad \abs{\nabla \xi_R}=\frac{r_+ \sqrt{1-r_+^2}}{r_- \sqrt{1-r_-^2}} \, \textup{ along }\Gamma_{-} (R).
\end{equation*}
\end{itemize}
Moreover, the following identity follows away from the critical points of $\xi _{R}$:
\begin{equation}\label{relacionHessiano}
	\nabla^2 \xi_{R} (r) = \left(\frac{r^2 \alpha}{((1-r^2)\xi_R - \alpha)^2} \abs{\nabla \xi_{R}} - \xi \right)g_{\sp^2} - \frac{2 \alpha r^2}{((1-r^2) \xi_R - \alpha)^2} d \xi_{R} \otimes d \xi_{R}.
\end{equation}
\end{proposition}


\subsection{Geometric interpretation}
Let us consider the minimal catenoid $C_{\alpha, \omega}$, defined as the unique rotationally symmetric minimal surface around the $z-$axis, symmetric with respect to the plane $\set{z=-\omega }$ and necksize $\alpha$. This catenoid can be parameterized by
\begin{equation}\label{parametrizationCatenoids}
	\psi_{\alpha, \omega} ( r, \theta ) = \alpha  \left( \frac{ \cos  \theta}{\sqrt{1- r^2}}  ,  \frac{\sin \theta }{\sqrt{1-r^2}} , {\rm arctanh} (r) - \omega \right) , \,\, r \in  (-1,1) \text{ and } \theta \in [0, 2\pi) ,
\end{equation}
with outward unit normal 
$$ N(r,\theta) = \left( \sqrt{1- r^2} \, \cos  \theta  , \sqrt{1- r^2} \, \sin \theta , - r \right) \in \s^2 \setminus \set{{\bf s}, {\bf n}}.$$

The support function $\xi (r,\theta ) = \meta{\psi (r ,\theta)}{N((r, \theta)}$ is then given by 
$$ \xi (r, \theta) = \alpha  (1- r \, {\rm arctanh}(r) + \omega r ) .$$

The coordinates $(r ,\theta) \in (-1,1) \times [0, 2 \pi )$ can be seen as a parametrization of $\s^2 \setminus \set{{\bf s}, {\bf n}}$ via the unit normal $N$. As we noted in the introduction, the support function $\xi$ satisfies 
$$\Delta \xi + 2 \xi =0 \text{ in } \s^2 \setminus \set{{\bf s}, {\bf n}}.$$

Thus, the rotationally symmetric solutions in annular domains described above are nothing but the support function of a given catenoid in $\r ^3$ (up to scaling and vertical translation) defined in the appropriate spherical domain by the Gauss map.  Now we can use this correspondence to derive some geometric properties of the functions defined in the previous section.

On the one hand, note that the zero level set of $\xi$ gives us the points of $C_{\alpha, \omega}$ where the position vector $\psi(r, \theta) = |\psi(r, \theta) |$ is orthogonal to the normal vector $N(r, \theta)$. Note also that $\xi$ depends only on $r$, so the image via $\psi$ of the zero level set of $\xi$ on $C_{\alpha, \omega}$ are horizontal circles contained in spheres (of different radii except for $\omega =0$) centered at the origin, and $C_{\alpha, \omega}$ intersects these spheres orthogonally. 

On the other hand, note that
\[
\lim_{r \to \pm 1} N(r,\theta)= \mp (0,0,1) \quad \textup{and} \quad \lim_{r \to \pm 1} \sqrt{1-r^2} \psi(r,\theta)=\alpha(\cos \theta, \sin \theta, 0),
\]
and it is clear that the function $f(r):= \sqrt{1-r^2} \left({\rm arctanh} (r) - \omega\right)$ satisfies $f((-1,1)) = (-1,1)$. Thus, we can deduce geometrically that there exist uniques $-1<r_-<0<r_+<1$ such that $\xi(r_+)=\xi(r_-)=0$. 

Let us denote the image in $C_{\alpha, \omega}$ of the two components of the zero level set of $\xi$ as $\zeta_{\pm}=\{\psi (r_{\pm}, \theta) ~\colon~ \theta \in [0,2\pi)\} $. Now, observe that
\begin{equation}\label{correspondenceGradient}
\abs{\psi (r_{\pm}, \theta)}^2 = \alpha^2 \left(\frac{1}{1-r_{\pm}^2}+\left(\textup{artanh}(r_{\pm})-\omega\right)^2\right)=\frac{\alpha^2}{r_{\pm}^2 \left(1-r_{\pm}^2\right)},
\end{equation}
so taking $\alpha =r_{+}\sqrt{1-r_+^2}$ we obtain that $\zeta_+$ is contained in the unit sphere $\mathbb{S}^2 = \mathbb{S}^2(0,1)$. On the other hand, it easy to check that the function $f(R)=r_{+}^2(R)(1-r_+^2 (R))/r_{-}^2(R)(1-r_-^2(R))$ with $R \in [0,1)$ has derivative
\[
f'(R)=\frac{4 r_{+}^2(R)(1-r_+^2(R))r_{-}^2(R)(1-r_-^2(R))}{(1-R^2)^2 r_{-}^4(R)(1-r_-^2(R))^2} \left( (r_+-2r_+^3)-(r_- - 2r_-^3) \right),
\]
so taking into account that $\bar r >0.8$ it can be easily checked that $f'(R)<0$ for all $R \in [0,1)$. We deduce then that there exist $\tilde r \leq 1$ such that $\zeta_- \in \mathbb{S}^2 (\tilde r)$, where $\mathbb{S}^2 (\tilde r)$ denotes the euclidean sphere centered at the origin and radius $\tilde r$.

 Thus, we conclude that solutions $\xi_R$ described in Proposition \ref{prop_2.1} correspond to support functions of pieces of catenoids
\begin{equation}\label{modelCatenoids}
C_R:= \left\{\psi_{\alpha (R), \omega (\alpha (R))} (r, \theta) ~\colon~  r \in  (r_{-} (R),r_{+}(R)), \, \theta \in [0, 2\pi)\right\},
\end{equation}
where
\[
\alpha (R)=r_+ (R)\sqrt{1-r_+^2 (R)} \quad \textup{and} \quad \omega(R) = \textup{arctanh}(R)+\frac{R}{1-R^2}, 
\]
and $\psi_{\alpha (R), \omega (R)}$ is given by \eqref{parametrizationCatenoids}. It holds that $C_R \subset \mathbb{B}^3$ for each $R \in [0,1)$. Also,  since $- r_-(0)=r_+(0) = \bar r$, it is clear that $C_0$ corresponds to the critical catenoid, i.e. the unique free boundary minimal catenoid in the unit sphere. Note that  $N(C_R)=\Omega_{-R}$ is the reflection with respect to the plane $\{z=0\}$ of $\Omega_{R}$ for each $R \in [0,1)$.

\begin{figure}[h!]
	\centering
	\begin{subfigure}[b]{0.42\linewidth}
		\includegraphics[width=\linewidth]{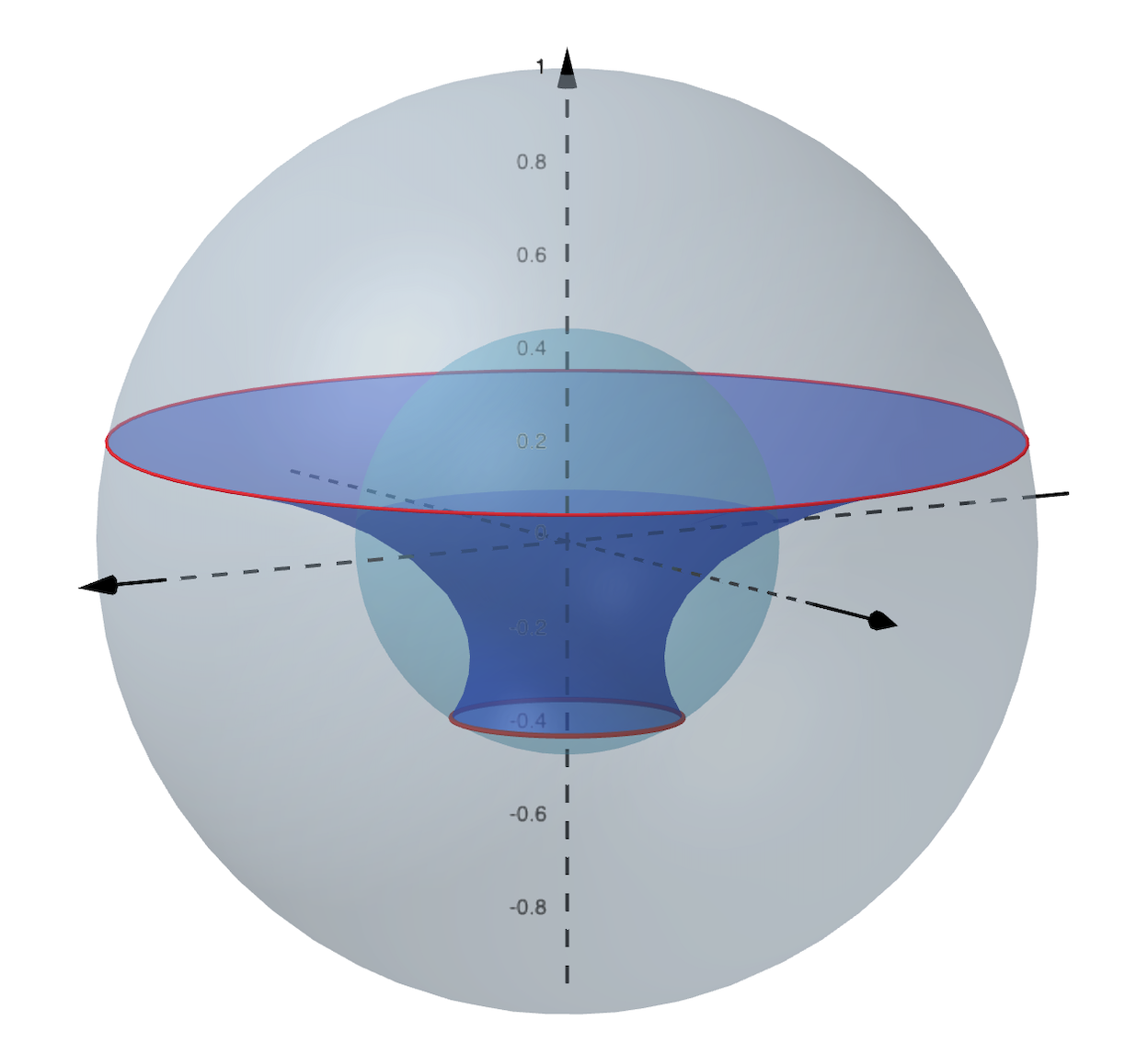}
		\label{fig:ModelCatenoid}
	\end{subfigure}
	\begin{subfigure}[b]{0.45\linewidth}
		\includegraphics[width=\linewidth]{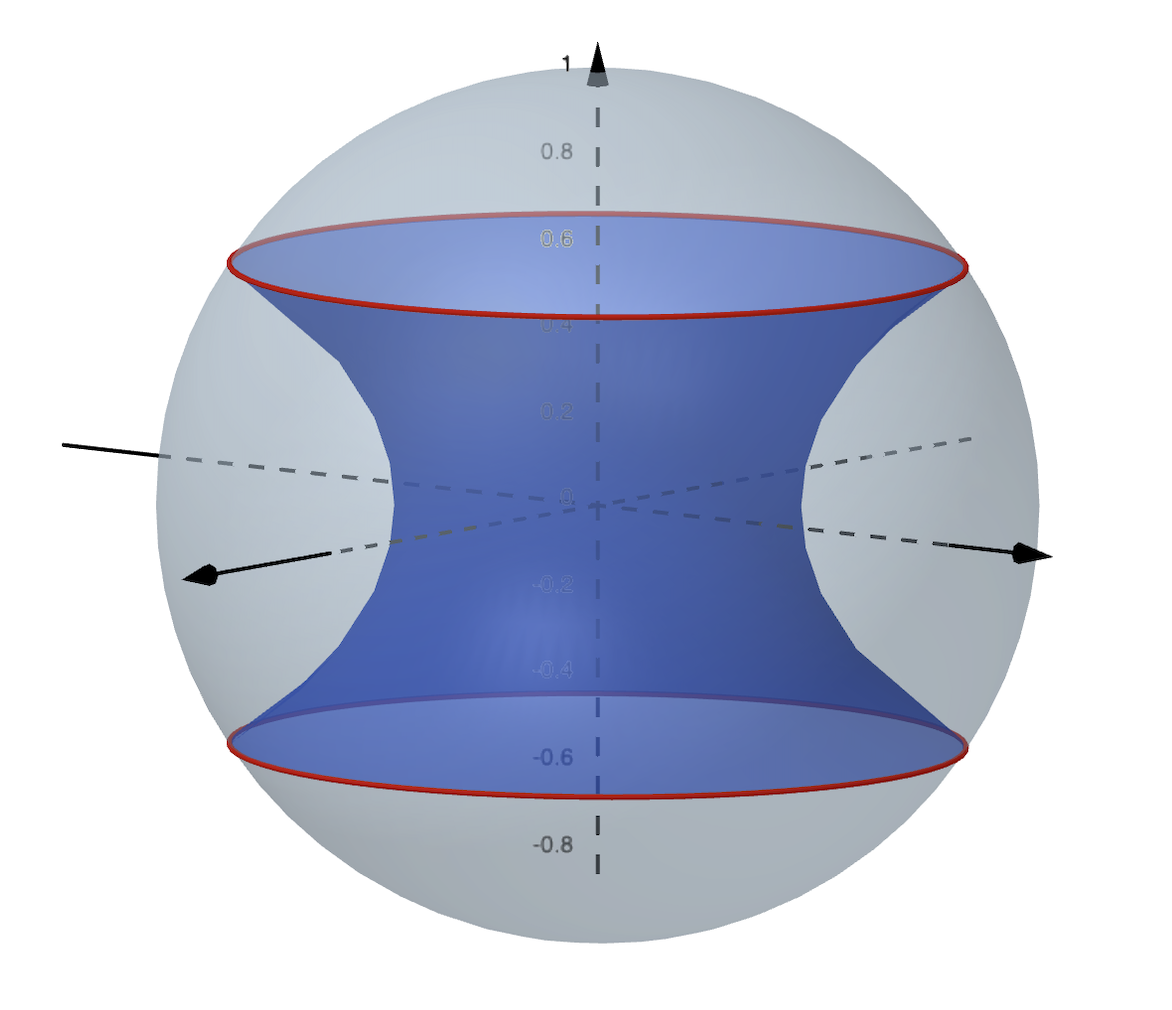}
		\label{fig:CriticalCatenoid}
	\end{subfigure}
	\caption{Here, two model catenoids defined in \eqref{modelCatenoids}. Left: model catenoid with parameter $R=1/2$. It can be seen that $\zeta_+$ is contained in the unit sphere and $\zeta_-$ is contained in a sphere of smaller radius. Right: critical catenoid, which corresponds to $C_0$.}
	\label{fig:ModelCatenoids}
\end{figure}


\section{Normalized Wall Shear Stress}\label{sec_3}
From now on, $\Omega$ will always denote a finite type region in the two-sphere.
\begin{definition}\label{finiteTypeDomain}
	Let $\Omega \subset \mathbb{S}^2$ be a domain with $\mathcal{C}^2$ boundary. Then we say that $\Omega$ is \textup{of finite type} if there exist a finite number of disjoint simply connected domains $D_1, \dots, D_k \subset \mathbb{S}^2$ such that 
	\begin{equation}
		\Omega = \mathbb{S}^2 \setminus \left(\cup_{i=1}^k \textup{cl}( D_i )\right).
	\end{equation}
\end{definition}
If $\Omega$ is a finite type domain with $k \geq 1$ boundary components, we will always write $\partial \Omega = \Gamma_1 \cup \dots \cup \Gamma_k$, being $\Gamma_i=\partial D_i$ for $i \in \{1,\dots, k\}$.
 
Following the ideas in \cite{ABM}, in this section we will classify solutions $(\Omega, \xi)$ to the eigenvalue problem 
\begin{equation}\label{DPS}
	\left\{	\begin{matrix}
		\Delta \xi + 2 \xi=0 & \text{ in } & \Omega ,\\[2mm]
		\xi >0 & \text{ in } & \Omega  ,\\[2mm]
		\xi = 0 & \text{ along } & \partial \Omega,
	\end{matrix}\right.
\end{equation}in terms of its \textit{normalized wall shear stress}, a scale-invariant quantity. By the maximum principle, a solution $(\Omega , \xi)$ to \eqref{DPS} does not have interior minimums. Hence, the set of interior critical points
$$ {\rm Crit}(\xi) := \set{ p \in \Omega \, : \, \, \nabla \xi (p) =0 } $$contains only saddle and maximum points. Given a solution $(\Omega , \xi)$ to \eqref{DPS} we denote by 
$$ {\rm Max} (\xi) := \set{ p \in \Omega \, : \, \, \xi (p) =\xi_{\textup{max}} :=\max_{\textup{cl}(\Omega)} \xi },$$and it is clear that this set must be non-empty. We will refer to ${\rm Max} (\xi) $ as the {\it top level set} of $\xi$.

\begin{remark}\label{RemEmpty}
If a solution $(\Omega , \xi) $ to \eqref{DPS} has a connected component $\mathcal{U} \subset \Omega \setminus \textup{Max} (\xi)$ such that $\textup{cl}(\U) \cap \partial \Omega = \emptyset$, then either $\xi = \xi_{\textup{max}}$ in $\mathcal{U}$ or there exists an interior minimum in $\mathcal{U}$, a contradiction in any case by the maximum principle.
\end{remark}
We also observe a structure result for the top level set:
\begin{lemma}\label{lemmaCritic}
	Let $(\Omega, \xi)$ be a solution to \eqref{DPS}, $\Omega \subset \mathbb{S}^2$ of finite type. Suppose that $\xi$ has infinite points of maximum, i.e., $\# \textup{Max}(\xi)=+\infty$. Then 
	$$\textup{Max}(\xi)=\gamma^0 \cup \gamma^1,$$where $\gamma^0$ is a finite set (could be empty) of points and $\gamma^1$ is a finite set of disjoint analytic closed curves. Moreover, for each $\gamma \in \gamma^1$ it holds that $\Omega \setminus \gamma = \Omega ^\gamma_1 \cup \Omega ^\gamma_2$ with $\partial \Omega ^\gamma_i \cap \partial \Omega \neq \emptyset$ for $i=1,2$.
\end{lemma}
\begin{proof}
We use the same arguments as in the proof of \cite[Theorem D]{ABM}. Solutions to \eqref{DPS} are real analytic (cf. \cite{Nir}), so by the Lojasiewicz structure theorem it follows that $\textup{Max}(\xi)=\gamma^0 \cup \gamma^1$, where $\gamma^0$ is a set of isolated points and $\gamma^1$ is a set of analytic curves. As $\gamma^0 \subset \mathbb{S}^2$ must be compact, it follows that $\gamma^0$ is a finite set of isolated points, and by \cite[Corollary 3.4.]{Chr} we get that $\gamma^1$ must be a finite set of analytic closed curves, possibly intersecting at a finite set of points. Furthermore, the curves of $\gamma^1$ are disjoint since, at the intersection points of two curves, the hessian of $\xi$ must vanish, which contradicts \eqref{DPS}. 
	
Finally, Remark \ref{RemEmpty} implies that $\gamma \in \gamma^1$ is not contractible in $\Omega$, that is, $\Omega \setminus  \gamma$ cannot contain a topological disk component $D$ such that $\partial D \cap \partial \Omega = \emptyset$. Hence, $\Omega \setminus \gamma = \Omega ^\gamma_1 \cup \Omega ^\gamma_2$ with $\partial \Omega ^\gamma_i \cap \partial \Omega \neq \emptyset$ for $i=1,2$, as claimed.
\end{proof}

The above lemma motivates the following 
\begin{definition}\label{DefPartition}
Let $(\Omega , \xi) $ be a solution to \eqref{DPS} and assume that $\gamma \subset \textup{Max}(\xi)$ is a closed embedded curve. Then, $\Omega \setminus \gamma = \Omega_1^\gamma \cup \Omega_2^\gamma$ and we will say that $\Omega_i^\gamma$, $i=1,2$, is the \textit{partition with respect to $\gamma$}.
\end{definition}

Now, we are ready to introduce the definition of \textup{normalised wall shear stress} (NWSS):

	\begin{definition}
		Let $(\Omega, \xi)$ be a solution to \eqref{DPS} and let $\Gamma \in \pi_{0} (\partial \Omega)$ be a connected component of the boundary. We define the \textup{normalised wall shear stress} (NWSS) of $\Gamma$ as 
		\begin{equation}\label{NWS}
			\overline{\tau}(\Gamma):= \frac{\max_{\Gamma} \abs{\nabla \xi}}{\xi_{\textup{max}}}.
		\end{equation}
	If now $\mathcal{U}$ is a connected component of $\Omega \setminus \textup{Max} (\xi)$, $\partial \Omega \cap \textup{cl}(\U) \neq \emptyset$, we define the NWSS of $\mathcal{U}$ as
	\begin{equation}\label{NWSSComponent}
			\overline{\tau}(\mathcal{U}):= \max \left\{ \overline{\tau}(\Gamma) ~\colon~ \Gamma \in \pi_{0} ( \partial \Omega \cap \textup{cl}(\U) ) \right\}.
	\end{equation}
Otherwise, we set $\overline{\tau}(\mathcal{U})=0$.
	\end{definition}

The rigidity results proved in the rest of the paper are derived by comparing geometric quantities of a fixed solution to the eigenvalue problem $(\Omega, \xi)$ with those of a certain model solution $(\Omega_{R}, \xi_{R})$; which motivates the following:

\begin{definition}
We say that a solution $(\Omega , \xi)$ to \eqref{DPS} is equivalent to a model solution $(\Omega _R , \xi _R)$, for some $R \in [0, 1]$, if they differ up to a rotation and a dilation. In such case, we will denote $(\Omega , \xi) \equiv (\Omega _R, \xi _R)$.
\end{definition}

In order to choose the appropriate model solution to compare with $(\Omega, \xi)$, we will use the NWSS defined previously.


\subsection{NWSS on the model solutions}

The NWSS in our model solutions at each of the components of the boundary can be computed in terms of the parameter $R$. In fact, 
\begin{equation}\label{NWSSModelSolutions}
	\overline{\tau}_{\pm} (R) :=\overline{\tau} \left(\Gamma_{\pm} (R)\right)= \pm \frac{1-R^2}{ r_{\pm} (R) \sqrt{1- r_{\pm}(R)^2}}, \quad R \in [0,1)
\end{equation}
where $\Gamma_{+}(R)$ and $\Gamma_{-}(R)$ are defined in \eqref{boundaryComponents}. Moreover, it can be shown that $\overline{\tau}_{-}$ and $\overline{\tau}_{+}$ are invertible functions. Set
\begin{equation}\label{criticalTau}
	\overline{\tau}_{\pm}(0)=\frac{1}{\bar{r}\sqrt{1-\bar{r}^2}}=:\tau_0 >1 \quad \left( \textup{where} \quad 1-\bar{r}\, \textup{arctanh}(\bar{r})=0 \right) ,
\end{equation}where, recall, $ r _+ (0) = - r_- (0) =  \bar r >0$.
\begin{lemma}\label{LemTau}
	The functions $\overline{\tau}_{+}: [0,1) \rightarrow [\tau_0, +\infty)$ and $\overline{\tau}_{-}: [0,1) \rightarrow (1 ,\tau_0 ]$ defined by \eqref{NWSSModelSolutions} are increasing and decreasing respectively.
\end{lemma}
\begin{proof}

	First, remember that in the previous section we proved
	\begin{equation}\label{aux1}
			\text{arctanh}( r_{\pm})-\frac{1}{ r_{\pm}}=\text{arctanh}(R)+\frac{R}{1-R^2}, \quad \forall R \in [0,1),
	\end{equation}
	so we have that $ r_{+} (R) \rightarrow 1$ and $ r_{-} (R) \rightarrow 0$ as $R \rightarrow 1$, and then by \eqref{derivH} we get  
	$$ r_{-}: [0,1) \rightarrow [-\bar{r},0) \quad \textup{and} \quad  r_{+}: [0,1) \rightarrow [\bar{r},1)$$ 
	are increasing functions of $R$. 
	\begin{quote}
		{\bf Claim A:} $\overline{\tau}_{+}$ and $\overline{\tau}_{-}$ are monotonic.
	\end{quote} 
	\begin{proof}[Proof of Claim A]
		We will see that the derivative of both functions cannot vanish. By contradiction, suppose there exists $ \bar R \in (0,1)$ such that $\overline{\tau}_{\pm}' ( \bar R)=0$. Then, by \eqref{NWSSModelSolutions}, we get
		$$2 \bar R \, r_{\pm}(\bar R)(1- r_{\pm}(\bar R)^2)+(1-2  r_{\pm}(\bar R)^2) (1-\bar R^2)  r_{\pm}'(\bar R)=0$$and hence
		\begin{equation}\label{auxdhpm}
			\frac{\partial  r_{\pm}}{\partial R} (\bar R) = \frac{2 (1- r_{\pm}(\bar R)^2)  r_{\pm} (\bar R)\bar R  }{(1-\bar R^2)(2  r_{\pm}(\bar R)^2-1)},
		\end{equation}
		so, comparing to \eqref{derivH}, we get
		\begin{equation}\label{auxhpm}
			 r_{\pm}(\bar R)(2  r_{\pm}(\bar R)^2-1)+(1-\bar R^2) \bar R =0.
		\end{equation}
		\begin{itemize}
			\item Assume first that $\bar \tau ' _+(\bar R) =0$. Since $ r_{+}(R)\geq \bar r > 0.8$ for all $R\in [0,1]$, then $2  r_{+}^2 (\bar R)-1 >0$, which contradicts \eqref{auxhpm}.
			
			\item Second, assume that $\bar \tau ' _-(\bar R) =0$. Then, \eqref{auxhpm} implies that $2  r_{-}(\bar R)^2-1 >0$ and, by \eqref{auxdhpm}, we get $\frac{\partial  r_{-}}{\partial R} (\bar R) <0$ since $ r_{-}(\bar R) < 0$, which contradicts \eqref{derivH}.
		\end{itemize}
		
		This proves Claim A.
	\end{proof}
	
	Next, we study the limit of $\overline{\tau}_{\pm}$ as $R$ goes to one. From \eqref{aux1} we obtain 
	$$ \frac{1+ r_{\pm}(R)}{1+R}(1-R)=(1- r_{\pm}(R))\textup{exp}\left( \frac{2}{ r_{\pm}(R)}+\frac{2 R}{1-R^2} \right), \quad \forall R \in [0,1)$$that yields
\begin{equation*}
\lim_{R \to 1^{-}}  (1- r_{\pm}(R))\textup{exp}\left( \frac{2}{ r_{\pm}(R)}+\frac{2 R}{1-R^2} \right) =0.
\end{equation*}
Since $r_+ (R) \to 1 $ as $R \to 1$ and $e^x \geq x^2$ for all $x \geq 0$, we get 
$$ 0=  \lim_{R \to 1^{-}} \abs{(1-r_+ ^2 (R)) {\rm exp}\left( \frac{1}{1-R^2}\right) }\geq \lim_{R \to 1^{-}} \abs{ \frac{1-r_+ ^2 (R)}{(1-R^2)^2} } $$and we conclude that $\overline{\tau}_{+} \rightarrow +\infty$ as $R \rightarrow 1$.

On the other hand, using again \eqref{aux1} and since $r_- (R) \to 0 $ as $R \to 1$ it is easy to conclude that
	\[
	\begin{split}
	1 &=\lim_{R \to 1^-} \frac{\textup{arctanh}(r_- (R))-\frac{1}{r_- (R)}}{\textup{arctanh} (R) + \frac{R}{1-R^2}}= \lim_{R \to 1^-}- \frac{\frac{1}{r_- (R)}}{\textup{arctanh} (R) + \frac{R}{1-R^2}}  \\
	&= \lim_{R \to 1^-} - \frac{1-R^2}{r_- (R) \left((1-R^2) \textup{arctanh}(R) + R \right)}=  \lim_{R \to 1^-}- \frac{1-R^2}{r_- (R)}, 
	\end{split}
	\]
	and by the last equality we conclude that $\overline{\tau}_- (R) \to 1$ as $R \to 1$. This finishes the proof.
\end{proof}

Looking at the previous lemma, one can figure out how to choose a model solution associated to an arbitrary solution to the eigenvalue problem $(\Omega,\xi)$: fixed $\mathcal{U} \subset \Omega \setminus \textup{Max} (\xi)$, if we have that $\overline{\tau}(\mathcal{U})>1$ then there exists a unique model solution $(\Omega_{R},\xi_R)$ such that $\overline{\tau}(\mathcal{U})=\overline{\tau}(\mathcal{V})$ for one of the components $\mathcal{V}$ of $\Omega_R \setminus \textup{Max} (\xi_R)$. This a priori is not a good correspondence, as one could have a solution to the eigenvalue problem with NWSS less or equal to one in a connected component of $ \Omega \setminus \textup{Max} (\xi)$. The key fact is that then one can prove that the solution must be $(\Omega_{1}, \xi_1)$, so we can rule out this case. To prove that we will need two technical results, which will be proven using $P$-functions.

\subsection{Associated $P$-function}
We will find a $P$-function associated with \eqref{DPS}, that is, a subharmonic function associated with a solution to \eqref{DPS}. The next result does not need to impose boundary conditions. 
\begin{proposition}
Let $\Omega \subset \s^2$ be a smooth domain and $\xi : \Omega \to \r $ be a smooth solution to $\Delta\xi + 2\xi=0$. Then, the $P$-function $\mathcal P :=   |\nabla  \xi| ^2+ \xi ^2 $ satisfies 
\begin{equation}\label{DivXi}
\Delta  \mathcal P  \geq 0 \text{ in } \Omega .
\end{equation}
\end{proposition}
\begin{proof}
By the Bochner-Weitzenb\"{o}ck formula, a solution $\xi$ to \eqref{DPS} satisfies
\begin{equation*}
\frac{1}{2}\Delta \abs{\nabla\xi}^2 = \abs{\nabla ^2  \xi}^2 - \abs{\nabla\xi}^2,
\end{equation*}and a standard computation shows
\begin{equation*}
\frac{1}{2}\Delta \xi ^2 = \xi \Delta\xi + \abs{\nabla \xi}^2 = - \frac{(\Delta \xi)^2}{2} +  \abs{\nabla \xi}^2.
\end{equation*}

Thus, summing up the above two equations, and the Geometric-Arithmetic Inequality, we obtain 
$$ \frac{1}{2} \Delta \mathcal P =   \abs{\nabla ^2  \xi}^2 - \frac{(\Delta \xi)^2}{2}  \geq 0 ,$$that is, \eqref{DivXi} holds. 
\end{proof}

As a consequence of the above $P$-function, we can obtain a characterization of disk-type solutions to \eqref{DPS} in terms of the normalized wall shear stress (NWSS) of the boundary components. Specifically

\begin{theorem}\label{theor_3.1}
Let $\Omega \subset \s^2$ be a smooth domain and $\xi : \Omega \to \r $ be a non-vanishing smooth solution to $\Delta \xi + 2\xi=0$ in $\Omega$ and $\xi =0$ along $\partial \Omega$. Assume that 
\begin{equation}\label{CondTau}
\bar \tau (\Gamma) \leq 1 \text{ for all } \Gamma  \in \pi _0 (\partial \Omega),
\end{equation}then $\Omega$ is a geodesic disk and $\xi$ is rotationally symmetric. 
\end{theorem}
\begin{proof}
Observe that \eqref{CondTau} and $\xi =0 $ along $\partial \Omega$ imply that
$$ {\rm max}_{\partial \Omega} \mathcal P \leq \xi_{\textup{max}}^2 .$$

Let $x \in \Omega$ be a point where $\xi ^2$ attains its maximum, hence $\xi (x) \nabla \xi (x) =0$ and $\xi (x) \neq 0$, thus $|\nabla  \xi| (x) =0$ and $\mathcal P (x) = \xi_{\textup{max}}^2$. Therefore, the strong maximum principle implies that $\mathcal P = \xi_{\textup{max}}^2$ on $\overline \Omega$ and $\nabla^2 \xi = - \xi g_{\sp ^2}$; now it follows that $\Omega $ is a geodesic ball and $\xi $ is rotationally symmetric using the same argument as in the proof of Lemma 2.4 of \cite{CV}. 
\end{proof}

Now we will introduce an energy function along level sets of solutions to \eqref{DPS} and we will show an important property of such energy function when $\overline{\tau} (\mathcal U) \leq 1$. Given  $\mathcal{U} \in \pi_{0} ( \Omega \setminus \textup{Max}(\xi))$, we define the energy function $E :[ 0, \xi_{\textup{max}} ) \to \r $ given by 
 \begin{equation}\label{energyFunction}
 	 E(t) = \frac{1}{\xi_{\textup{max}}^2 -t^2} \int _{\textup{cl}(\U) \cap \set{\xi=t}} |\nabla \xi|.
 \end{equation}

Note that $E$ is continuous in regular values of $[0,\xi_{\textup{max}})$ since the Lebesgue integral is absolutely continuous.
\begin{lemma}\label{LemDecreasing}
Let $(\Omega , \xi )$ be a solution to \eqref{DPS}. Suppose that there exists $\mathcal{U} \in \pi_{0} (\Omega \setminus \textup{Max} (\xi))$ such that $\overline{\tau} (\mathcal{U}) \leq 1$. Then $E$ is non-increasing.
\end{lemma}
\begin{proof}
Given $\epsilon>0$, consider the inner region
 \begin{equation}\label{innerRegionEpsilon}
\mathcal U_{\epsilon}:= \left\{ p \in \mathcal U ~\colon~ \xi (p) < \xi_{\textup{max}}-\epsilon \right\} \subset \mathcal U.
 \end{equation}
Since the critical values of $\xi$ in $\mathcal U$ are isolated (cf.  \cite{Souc}), then we can choose $\epsilon$ small enough so that $\xi_\textup{max}-\epsilon$ is a regular value, and hence, the set $\mathcal{U} \cap \{\xi=\xi_{\textup{max}}-\epsilon\}$ is the disjoint union of a finite number of analytic curves by the Lojasiewicz structure theorem. Using \eqref{DivXi} and the maximum principle, we obtain that
\[
\max_{\mathcal{U}_{\epsilon}}\mathcal P \leq \max_{\partial \mathcal{U}_{\epsilon}} \mathcal P .
\]

On the one hand, $|\nabla\xi| \rightarrow 0$ as $\epsilon \rightarrow 0$ yields
\[
\lim_{\epsilon \to 0^+} \max_{\{\xi=\xi_{\textup{max}}-\epsilon\}}\mathcal P= \xi_{\textup{max}}^2. 
\]
On the other hand, from Remark \ref{RemEmpty} and $\overline{\tau} (\mathcal{U}) \leq 1$, we get $|\nabla\xi|^2 + \xi^2= |\nabla \xi|^2 \leq \xi_{\textup{max}}^2$ along $\partial \Omega \cap \textup{cl}(\U)$. Hence, we conclude that
\[
|\nabla \xi|^2 + \xi^2=\mathcal{P} \leq \xi_{\textup{max}}^2 \quad \textup{in} \quad \mathcal{U}. 
\]
Hence, using $\Delta\xi + 2\xi=0$ in $\Omega$, we obtain
\begin{equation*}
\begin{split}
{\rm div} \left( \frac{\nabla \xi }{\xi_{\textup{max}} ^2 - \xi ^2}\right) = \frac{2 \xi}{(\xi_{\textup{max}}^2 -\xi ^2)^2} \left( \mathcal P - \xi_{\textup{max}}^2\right) ,
\end{split}
\end{equation*}
and, since $\xi \geq 0$ in $\Omega$ and $ \mathcal P \leq \xi_{\textup{max}}^2$ in $\mathcal{U}$, we get 
$$  {\rm div}\left( \frac{\nabla  \xi }{\xi_{\textup{max}} ^2 - \xi ^2}\right) \leq 0 \quad \textup{in} \quad \mathcal{U}.$$

Now, we choose regular values $0 \leq  t_1 < t_2 < \xi_{\textup{max}}$ and we integrate the above inequality along the finite perimeter set $\{t_1\leq \xi \leq t_2\}$. Then, applying the Divergence Theorem:
\begin{equation*}
\begin{split}
	0 \geq & \int _{\set{t_1 \leq \xi \leq t _2}}  {\rm div}\left( \frac{\nabla \xi }{\xi_{\textup{max}} ^2 - \xi ^2}\right)                                       \\
	       & = \int _{\set{u=t_1}} \meta{ \frac{\nabla \xi }{\xi_{\textup{max}} ^2 - \xi ^2}}{n}+ \int _{\set{u=t_2} }\meta{ \frac{\nabla \xi }{\xi_{\textup{max}} ^2 - \xi ^2}}{n} \\
	       & = -\frac{1}{\xi_{\textup{max}}^2- t_1^2} \int _{\set{u=t_1}} |\nabla  \xi|+ \frac{1}{\xi_{\textup{max}}^2- t_2^2} \int _{\set{u=t_2}}  |\nabla \xi|                      \\
	       & = -E(t_1) + E(t_2),
\end{split}
\end{equation*}where we have used that $\nu =- \nabla \xi / |\nabla \xi |$ along $\set{u = t_1}$ and $\nu =  \nabla \xi / |\nabla \xi |$ along $\set{u = t_2}$ are the outer unit normal respectively. Therefore, 
$$E(t_1) \geq E(t_2)  \text{ for } 0 \leq  t_1 < t_2 < \xi_{\textup{max}} \text{ regular values}.$$ 

Since $E$ is continuous and the set of critical values is finite, then $E$ is a non-increasing function.
\end{proof}

\subsection{Expected critical height}

As said above, the NWSS on the model solutions will allow us to define an expected value for any solution to \eqref{DPS}. Specifically:

\begin{definition}\label{DefCriticalHeight}
Let $(\Omega, \xi)$ be a solution to \eqref{DPS} and $\mathcal U \in \pi _0 (\Omega \setminus \textup{Max} (\xi))$. Set $\tau_0$ as in \eqref{criticalTau}, if $\Omega$ is not a topological disk then we define the \textup{expected critical height of }$\mathcal U$ as follows:

\begin{itemize}
	\item if $\overline{\tau} (\mathcal U) < \tau_0$, we set
	\begin{gather}\label{sutilezaCriticalHeight}
		\bar R(\mathcal U)= \overline{\tau}_{-}^{-1} \left( \overline{\tau} (\mathcal U) \right),
	\end{gather}
	\item  if $\overline{\tau} (\mathcal U) \geq \tau_0$, we set
	\begin{gather}
		\bar R(\mathcal U)= \overline{\tau}_{+}^{-1} \left( \overline{\tau} (\mathcal U) \right).
	\end{gather}
\end{itemize}
If $\Omega$ is a topological disk, then we set $\bar R(\mathcal U)=1$.
\end{definition}

Observe that in \eqref{sutilezaCriticalHeight} we are assuming implicitly that, if $\Omega$ is not a topological disk, then $\overline{\tau} (\mathcal U)>1$ for each connected component $\mathcal U$ of $\Omega \setminus \textup{Max} (\xi)$. In order to prove that Definition \ref{DefCriticalHeight} is consistent we first show:

\begin{theorem}\label{ThmAux}
Let $( \Omega ,\xi )$ be a solution to \eqref{DPS} and $\mathcal U \in \pi_{0} (\Omega \setminus \textup{Max} (\xi))$. If $\overline{\tau}(\mathcal U) \leq 1$, then $(\Omega, \xi) \equiv (\Omega_1, \xi_{1})$.
\end{theorem}
\begin{proof}
By Lemma \ref{lemmaCritic}, $\textup{Max} (\xi)= \gamma^0 \cup \gamma^1 $, where $\gamma^1$ is the disjoint union of a finite number of analytic curves and $\gamma^0$ is a finite set of isolated points.

We first show the result if $\gamma^1 = \emptyset$. In such a case, $\Omega \setminus \textup{Max} (\xi)$ would be connected and Theorem \ref{theor_3.1} implies that $(\Omega, \xi)$ is a rotationally symmetric solution on a disk, so it must be $(\Omega_1, \xi_1)$ after a rescaling. Hence, we can assume that $\gamma^1$ is non-empty and is in contact with $\U$.

Since $\gamma^1 \cap \textup{cl}(\U)$ is a set of analytic closed curves, by the Lojasiewicz inequality (see \cite[Theorem 2.1]{Chr}), there exists a neighborhood $V$ of $\gamma^1 \cap \textup{cl}(\U)$, $V \cap \gamma^0 = \emptyset$, and two constants $c>0$ and $1/2 \leq \theta < 1$ such that 
$$\xi_{\textup{max}}-\xi < 1 \text{ and } \abs{\nabla \xi }(x) \geq c (\xi_{\textup{max}}-\xi (x))^{\theta}, \quad \forall x \in V.$$

Then, since $\{\xi=t\} \cap \textup{cl}(\U) \subset V$ as $t \rightarrow \xi_{\textup{max}}$, using the energy function \eqref{energyFunction} we get the following chain of inequalities
	\begin{equation*}
		\begin{split}
			\frac{1}{\xi_{\textup{max}}^2} \int _{\textup{cl}(\U) \cap \set{\xi=\xi_{\textup{max}}}} |\nabla \xi|  = E(0) \geq ^{\text{Lemma } \ref{LemDecreasing}}E(t) & = \frac{1}{\xi_{\textup{max}}^2 -t^2} \int _{\textup{cl}(\U) \cap \set{\xi=t}} |\nabla  \xi|  \\
			& \geq \frac{1}{\xi_{\textup{max}} -t} \int _{\textup{cl}(\U) \cap \set{\xi=t}} |\nabla  \xi|   \\
			& \geq \frac{c}{(\xi_{\textup{max}}-t)^{1-\theta}(\xi_{\textup{max}}+t)} \abs{\{\xi=t\} \cap\textup{cl}(\U)},
		\end{split}
	\end{equation*}
We can assume that all $t$'s in the previous inequalities are regular values of $\xi$, as we know by \cite[Theorem 1]{Souc} that if $Z$ is the set of critical points of $\xi$ then $\xi (K \cap Z)$ is finite for any compact set $K \subset \Omega$. Then $\abs{\{\xi=t\} \cap \textup{cl}(\U)}>0$. Letting $t \rightarrow \xi_{\textup{max}}$ we obtain that $E(0)=+\infty$, a contradiction.  
\end{proof}

Finally, we are ready to prove that Definition \ref{DefCriticalHeight} is well-posed:

\begin{theorem}\label{theor_CriticalHeight}
Let $( \Omega ,\xi )$ be a solution to \eqref{DPS} and $\mathcal U \in \pi_{0} (\Omega \setminus \textup{Max} (\xi))$. Then, the expected critical height of $\mathcal U$, $\bar R(\mathcal U)$, is well defined and takes values in $[0,1]$. Moreover, $\bar R(\mathcal U)=1$ if, and only if, $(\Omega, \xi) \equiv (\Omega_1, \xi_{1})$.
\end{theorem} 
\begin{proof}
First, from Remark \ref{RemEmpty}, any $\mathcal U \in \pi_{0} (\Omega \setminus \textup{Max} (\xi))$ must satisfy $\partial \Omega \cap \overline{\mathcal U} \neq \emptyset$ and, from \eqref{NWS}, $\overline \tau (\mathcal U) \geq 0$. Observe that, by definition, $R(\mathcal U) = 1$ if, and only if, $\overline \tau (\mathcal U) \in [0,1]$. Then, Theorem \ref{ThmAux} implies that $(\Omega, \xi)=(\Omega_1, \xi_{1})$ up to a rotation and a dilation. Therefore, there exists a bijective correspondence between $\bar R(\mathcal U) \in [0 ,1 )$ and $\bar \tau (\mathcal U) \in (1, + \infty)$ by Lemma \ref{LemTau}, which finishes the proof. 
\end{proof}


\section{Comparison geometry}\label{sec_4}
Now, using the expected critical height we can associate to each general solution $(\Omega, \xi)$ one of the model solutions defined in Section \ref{sec_2}.
\begin{definition}\label{associatedSolution}
	Let $(\Omega, \xi)$ be a solution to \eqref{DPS}, $\mathcal U \in \pi _0 \left(\Omega \setminus \textup{Max} (\xi) \right)$ and set $\bar R :=  R(\mathcal U)$ the expected critical height of $\mathcal U$. Then we say that $(\Omega_{\bar R}, \xi_{\bar R})$ is the \textup{associated model solution} to $\xi$ inside $\mathcal{U}$.
\end{definition} 
\begin{remark}\label{remarkNotation}
	From now on, we omit the parameter $\bar R$ and we only write $(\bar \Omega, \bar \xi ) := (\Omega_{\bar R}, \xi_{\bar R})$, we also denote the extremal of the interval of definition and the scale parameter of $\bar \xi$ as $\bar r_{\pm }:= r_{\pm}(\bar R)$ and $\bar \alpha=\alpha (\bar R)$ respectively, since $\bar R$ is already fixed. 
\end{remark}
In this section, we will construct a function to compare the level sets of $\xi$ with those of our model solution. To construct the comparison function, we need our general function $\xi$ to be normalized.

\begin{definition}\label{normalizedSolution}
Let $(\Omega, \xi)$ be a solution to \eqref{DPS}, $\mathcal U \in \pi _0 \left(\Omega \setminus \textup{Max} (\xi) \right)$ and set $\bar R :=  R(\mathcal U)$ the expected critical height of $\mathcal U$. Then we will say that the solution $\xi$ is \textup{normalized} if $\xi_{\textup{max}}=(\xi_{R})_{\textup{max}}$.
\end{definition}
Let us suppose that $\xi$ is normalized, consider the function $F: [0, \xi_{\textup{max}}] \times [\bar r_{-},\bar r_{+}] \to \r$ given by 
$$ F(\xi , r) = \xi - \bar \alpha(1-r \, {\rm arctanh} (r)+\omega r) ,$$ where $\omega$ is given by \eqref{rparameter} and $\bar \alpha$ is given in Proposition \ref{prop_2.1}. Hence
$$\bar \alpha^{-1} \frac{\partial F}{\partial r} (r)=    {\rm arctanh} (r) + \frac{r}{1-r^2}-r.$$and therefore $ \frac{\partial F}{\partial r} =  0 $ if, and only if, $r =\bar R$. As a consequence of the Implicit Function Theorem, there exist two smooth functions
$$\chi _{-} :  [0, \xi _{\max} ] \to [\bar r_{-},\bar R] \quad \textup{and} \quad \chi _{+} :  [0, \xi _{\max} ] \to [\bar R, \bar r_{+}] $$ such that $$ F(\xi , \chi _\pm (\xi)) =0 \quad \textup{for all} \quad \xi \in [0 , \xi_{\textup{max}}].$$ 

Now we can define a \textit{pseudo-radial function} following \cite[Definition 3]{ABM}.

\begin{definition}\label{psudoRadialFunctions}
	Let $\mathcal U \subset \Omega \setminus \textup{Max}(\xi)$ be a connected component and set $\tau_0$ as in \eqref{criticalTau}. Then: 
	\begin{itemize}
		\item If $\overline{\tau}(\mathcal U) \geq \tau_0$, we define the pseudo-radial function $\Psi_{+}$ associated with the region $\mathcal U$ as 
		\begin{equation}\label{pseudo-positive}
			\begin{split}
				\Psi_{+}:  \, & \mathcal U \to [\bar R, \bar r_{+}] \\
				& \, p \,\, \mapsto \Psi_{+} (p):= \chi_{+}\left(\xi (p)\right).
			\end{split}
		\end{equation}
       \item If $\overline{\tau}(\mathcal U) < \tau_0$, we define the pseudo-radial function $\Psi_{-}$ associated with the region $\mathcal U$ as 
\begin{equation}\label{pseudo-negative}
	\begin{split}
		\Psi_{-}:  \, & \mathcal U \to [\bar r_{-},\bar R] \\
		& \, p \,\, \mapsto \Psi_{-} (p):= \chi_{-}\left(\xi (p)\right).
	\end{split}
\end{equation}
\end{itemize}
\end{definition}

\begin{remark}
Note that if $\mathcal U \subset \Omega \setminus \textup{Max}(\xi)$ is such that $\overline{\tau}(\mathcal{U})=\tau_0$ then we can define the pseudo-radial function associated to $\mathcal{U}$ as $\Psi_{+}: \mathcal{U} \to [0, \bar r]$ or as $\Psi_{-}: \mathcal{U} \to [-\bar r,0]$. We can use both definitions because the model solution $(\Omega_0, \xi_0)$ is symmetric with respect to the meridian $\{r=0\}$ in cylindrical coordinates in the sphere, and both components of $\Omega_0 \setminus \textup{Max}(\xi_0)=\Omega \setminus \{r=0\}$ have the same NWSS equal to $\tau_0$. To fix one, we choose to define the pseudo-radial function as \eqref{pseudo-positive} when $\overline{\tau}(\U)=\tau_0$.
\end{remark}

\begin{remark}
	In order to perform some formal computations using the pseudo-radial function, it is important to do previous considerations about the notation. First, we will denote $\Psi_{\pm}=\Psi$ and $\chi_{\pm}=\chi$ and we will do the computations considering both possibilities at the same time. On the other hand, note that $\chi$ is a function of $\xi$ as a real parameter, while $\Psi$ is a function defined on a domain in the sphere. Also, we denote the derivatives with respect to $\xi$ with a dot. For example, $\dot{\chi}$ will denote the derivative of $\chi$ with respect to the parameter $\xi$.
\end{remark}

A straightforward computation gives
$$ 0 = \frac{d F}{d \xi } (\xi , \chi _\pm (\xi)) = \frac{\partial F}{\partial \xi } (\xi ,\chi _\pm (\xi)) + \frac{\partial F}{\partial r } (\xi , \chi _\pm (\xi)) \dot{\chi}_{\pm}  (\xi) $$and 
$$ 0 = \frac{d^2 F}{d \xi^2 } (\xi , \chi _\pm (\xi)) = \frac{\partial ^2 F}{\partial r^2 } (\xi , \chi _\pm (\xi)) \dot{\chi}_{\pm}  (\xi) ^2 + \frac{\partial F}{\partial r } (\xi , \chi _\pm (\xi))\ddot{\chi}_{\pm}(\xi) $$where we have used that $\frac{\partial F}{\partial \xi } (\xi ,\chi _\pm (\xi)) = 1$ and $\frac{\partial^2 F}{\partial \xi ^2} (\xi , \chi _\pm (\xi)) = 0 =\frac{\partial ^2 F}{\partial r \partial \xi } (\xi ,\chi _\pm (\xi))$. We get 
\begin{equation}\label{primeraDerivadaChi}
	\dot{\chi}_{\pm}  (\xi) = -\frac{\partial F}{\partial r } (\xi , \chi _\pm (\xi))  ^{-1}=\mp \sqrt{1-\chi _\pm (\xi)^2} \abs{\nabla \bar \xi}(\chi _\pm (\xi))^{-1}
\end{equation}
and 
\begin{equation}\label{segundaDerivadaChi}
	\ddot{\chi}_{\pm}(\xi) = - \frac{\partial ^2 F}{\partial r^2 } (\xi , \chi _\pm (\xi)) \frac{\partial F}{\partial r } (\xi , \chi _\pm (\xi))^{-3}= \mp \frac{2\bar \alpha \abs{\nabla \bar \xi}(\chi _\pm (\xi))^{-3}}{\sqrt{1-\chi _\pm (\xi)^2}}.
\end{equation}
We can also compute
$$ \nabla  \Psi= \dot{\chi} \nabla \xi \, \text{ and } \, \nabla ^2  \Psi = \dot{\chi} \nabla ^2 \xi + \ddot{\chi} d \xi \otimes d \xi .$$

\subsection{Gradient estimates}
First, we will use the pseudo-radial function to compare the norm of the gradient of the solution $\xi$ along level sets with that of the model solution $\bar \xi$. Suppose that the function $\xi$ is normalized and consider the comparison function 
\begin{equation}\label{normModelGradient}
 W_{\bar R} = \abs{\nabla \bar \xi}^2 \circ \Psi= \frac{1-\Psi^2}{\Psi^2} \left(\xi-\frac{\bar  \alpha}{1-\Psi^2}\right)^2,
\end{equation}
and set
\begin{equation}\label{normGradient}
W= \abs{\nabla \xi}^2 \text{ in } \mathcal U. 
\end{equation}

Our aim is to prove the following result:
\begin{theorem}\label{theor_Gradient}
	Let $(\Omega, \xi)$ be a solution to \eqref{DPS}, $\mathcal U \in \pi _0 \left(\Omega \setminus \textup{Max}(\xi) \right)$ and set $\bar R = R (\mathcal U) \in [0,1)$ the expected critical height of $\mathcal U$. Let $(\bar \Omega, \bar \xi)$ be its associated model solution and assume that $\xi$ is normalized. Then, it holds
	\[
	W(p) \leq W_{\bar R}(p) \, \text{ for all } p \in \mathcal U,
	\]
moreover, if the equality holds at one single point of $\mathcal U$, then $(\Omega, \xi) \equiv (\bar\Omega, \bar \xi)$.
\end{theorem}

To prove the previous theorem we will need to know how the previous functions behaves near the top level set $\textup{Max}(\xi)$. The following lemma will be useful later (cf. \cite{ABM}).
\begin{lemma}\label{lemmaWR}
	Let $(\Omega, \xi)$, $(\bar \Omega, \bar \xi)$ and $\mathcal{U}$ be as in Theorem \ref{theor_Gradient} and consider $W_{\bar R}$ defined in \eqref{normModelGradient}. Then, if $p \in \textup{Max}(\xi)$, it holds:
	\[
	\lim_{x \in \mathcal{U},x \to p} \frac{W_{\bar R}}{\xi_{\textup{max}}-\xi} = 4 \xi_{\textup{max}}.
	\]
\end{lemma}
\begin{proof}
The proof of this lemma is a simple computation using the Taylor expansions of $W_{\bar R}$ and $\xi_{\textup{max}}-\xi$ as functions of $\Psi$. Observe that taking $z=\Psi-\bar{R}$ then it is clear that
	\begin{equation}\label{expansionXi}
		\begin{split}
			\xi_{\textup{max}}-\xi= \frac{\bar \alpha}{1-\bar{R}^2}-\xi &= \frac{\bar \alpha}{(1-\bar{R}^2)^2}(\Psi-\bar{R})^2+O((\Psi-\bar{R})^3) \\
			&= \frac{\bar \alpha}{(1-\bar{R}^2)^2}z^2+ O(z^3).
		\end{split}
	\end{equation}
Note that we have done the Taylor expansion of $\xi \circ \Psi$ at $\Psi=\bar{R}$, although $\Psi (p)$ is not defined in $\textup{Max}(\xi)$. However, it is clear that we can extend $\Psi$ to $\textup{Max}(\xi)$ by $\Psi(p)=\bar{R}$ for all $p \in \textup{Max}(\xi)$, so the previous expansion makes sense. 

On the other hand, observe that
\[
W_{\bar R} (\Psi)= \frac{1-\Psi^2}{\Psi^2} \left(\xi-\frac{ \bar \alpha}{1-\Psi^2}\right)^2=(1-\Psi^2) \left(\frac{\partial \xi}{\partial \Psi}\right)^2.
\]
Hence
\[
\frac{\partial W_{\bar R}}{\partial\Psi}(\bar R)=0 \text{ and }\frac{\partial^2 W_{\bar{R}}}{\partial \Psi^2}(\bar{R})=2 (1-\bar{R}^2)\left(\frac{\partial^2 \xi}{\partial \Psi^2} (\bar{R})\right)^2=\frac{8 \bar{\alpha}^2}{(1-\bar{R}^2)^3},
\]
so the Taylor expansion of $W_{\bar{R}}$ up to second order is
\begin{equation}\label{expansionWR}
W_{\bar{R}}= \frac{4 \bar{\alpha}^2}{(1-\bar{R}^2)^3}z^2+O(z^3).
\end{equation}
Then it is clear that
\[
\lim_{x \in \mathcal{U},x \to p} \frac{W_{\bar R}}{\xi_{\textup{max}}-\xi}=\lim_{z \to 0} \frac{\frac{4 \bar{\alpha}^2}{(1-\bar{R}^2)^3}z^2+O(z^3)}{\frac{\bar{\alpha}}{(1-\bar{R}^2)^2}z^2+ O(z^3)} = \frac{4 \bar{\alpha}}{1-\bar{R}^2}
\]
and the result is proved.
\end{proof}

Next, we shall establish some differential inequalities. Firstly, we want to bound the norm of the hessian of $\xi$ in terms of the functions defined in \eqref{normModelGradient} and \eqref{normGradient}. From \eqref{relacionHessiano}, we obtain the following expression:

\begin{eqnarray*}
	0 &\leq& \abs{ \nabla^2 \xi + \frac{2 \bar{\alpha} \Psi^2}{ \left( (1-\Psi^2)\xi - \bar{\alpha} \right)^2} d\xi \otimes d\xi+  \left(\xi-\frac{\Psi^2 \bar{\alpha}}{\left((1-\Psi^2)\xi-\bar{\alpha}\right)^2} \abs{\nabla \xi}^2\right) g_{\sp^2} }^2\\
	&=& \abs{\nabla^2 \xi}^2 +\frac{4 \bar{\alpha}^2 \Psi^4}{\left((1-\Psi^2)\xi-\bar{\alpha}\right)^4} \abs{d \xi \otimes d \xi}^2+ \left(\xi-\frac{\Psi^2 \bar{\alpha}}{\left((1-\Psi^2)\xi-\bar{\alpha}\right)^2} \abs{\nabla \xi}^2\right) \abs{g_{\sp^2}}^2 \\
	& & + \frac{4 \bar{\alpha} \Psi^2}{\left((1-\Psi^2)\xi-\bar{\alpha}\right)^2} \left(\xi-\frac{\Psi^2 \bar{\alpha}}{\left((1-\Psi^2)\xi-\bar{\alpha}\right)^2} \abs{\nabla \xi}^2\right) \pscalar{d \xi \otimes d \xi}{g_{\sp^2}} \\
	& & + \frac{4 \bar{\alpha} \Psi^2}{\left((1-\Psi^2)\xi-\bar{\alpha}\right)^2} \pscalar{\nabla^2 \xi}{d \xi \otimes d \xi}+2\left(\xi-\frac{\Psi^2 \bar{\alpha}}{\left((1-\Psi^2)\xi-\bar{\alpha}\right)^2} \abs{\nabla \xi}^2\right) \pscalar{\nabla^2 \xi}{g_{\sp^2}}.
\end{eqnarray*}
Using that
\[
\abs{d\xi \otimes d\xi}^2=\abs{\nabla \xi}^4, \,\,  \pscalar{d\xi \otimes d\xi}{g_{\sp^2}}=\abs{\nabla \xi}^2 \,\,  \textup{and} \,\,  \pscalar{\nabla^2 \xi}{g_{\sp^2}}= \Delta \xi,
\]
we get
\begin{eqnarray*}
	0 &\leq& \abs{\nabla^2 \xi}^2+\frac{4 \bar{\alpha} \Psi^2}{\left((1-\Psi^2)\xi-\bar{\alpha}\right)^2} \pscalar{\nabla^2 \xi}{d \xi \otimes d \xi}-2 \xi^2 + \frac{2 \bar{\alpha}^2 \Psi^4}{\left((1-\Psi^2)\xi-\bar{\alpha}\right)^4} \abs{\nabla \xi}^4  \\ & & + \frac{4 \xi \bar{\alpha} \Psi^2}{\left((1-\Psi^2)\xi-\bar{\alpha}\right)^2} \abs{\nabla \xi}^2, 
\end{eqnarray*}
and then we reach the inequality
\begin{equation}\label{ine1}
\begin{split}
		\abs{\nabla^2 \xi}^2  \geq & -\frac{4 \bar{\alpha} \Psi^2}{\left((1-\Psi^2)\xi-\bar{\alpha}\right)^2} \pscalar{\nabla^2 \xi}{d \xi \otimes d \xi} +2 \xi^2 \\
		 &- \frac{2 \bar{\alpha}^2 \Psi^4}{\left((1-\Psi^2)\xi-\bar{\alpha}\right)^4} \abs{\nabla \xi}^4  - \frac{4 \xi \bar{\alpha} \Psi^2}{\left((1-\Psi^2)\xi-\bar{\alpha}\right)^2} \abs{\nabla \xi}^2.
\end{split}
\end{equation}

To compute the scalar product $\pscalar{\nabla^2 \xi}{d\xi \otimes d\xi}$ we use the following lemma which is a standard computation in geodesic coordinates.
\begin{lemma}
Let $(M,g)$ be a $n$-dimensional Riemannian manifold and $f \in \mathcal{C}^{\infty} (M)$. Then:
\begin{equation}\label{technicalEquality}
\pscalar{\nabla \abs{\nabla f}^2}{\nabla f}= 2 \nabla^2 f (\nabla f, \nabla f) = 2 \pscalar{\nabla^2 f}{df \otimes df}.
\end{equation}
\end{lemma}

Note that
\[
\nabla W_{\bar R}=\frac{\partial}{\partial \Psi} \left(  \frac{1-\Psi^2}{\Psi^2} \left(\xi-\frac{\bar{\alpha}}{1-\Psi^2}\right)^2 \right) \dot{\chi}(\xi) \nabla \xi=-2 \left( \xi + \frac{\bar{\alpha}}{1-\Psi^2}\right)\nabla \xi, 
\]
and then using \eqref{technicalEquality} we get
\[
\pscalar{\nabla^2 \xi}{d\xi \otimes d\xi}= \frac{1}{2} \pscalar{\nabla(W-W_{\bar{R}})}{\nabla \xi}-\left( \xi + \frac{\bar{\alpha}}{1-\Psi^2}\right)\abs{\nabla \xi}^2.
\] 
We use the previous equality and \eqref{ine1} to get
\begin{equation}\label{ine2}
	\begin{split}
		\abs{\nabla^2 \xi}^2  \geq & -\frac{2 \bar{\alpha} \Psi^2}{\left((1-\Psi^2) \xi - \bar{\alpha}\right)^2}\pscalar{\nabla (W-W_{\bar R})}{\nabla \xi}+ 2\xi^2 \\
		& -\frac{2 \bar{\alpha}^2 \Psi^4}{\left((1-\Psi^2) \xi - \bar{\alpha}\right)^4} \abs{\nabla \xi}^4+\frac{4 \bar{\alpha}^2 \Psi^2}{(1-\Psi^2)\left((1-\Psi^2) \xi - \bar{\alpha}\right)^2} \abs{\nabla \xi}^2  .
	\end{split}
\end{equation}

Now, we want to apply the maximum principle to the function $W-W_{\bar R}$, so we want this function to satisfy an elliptic inequality. With this in mind, we compute the laplacian of this expression. Using the Bochner's formula,
\[
\Delta (W-W_{\bar R})= 2 \abs{\nabla^2 \xi}^2- 4 \left(\xi+\frac{\bar{\alpha}}{1-\Psi^2}\right)+\frac{4 \Psi^2 \bar{\alpha}}{(1-\Psi^2)\left((1-\Psi^2) \xi - \bar{\alpha}\right)} W.
\]
From \eqref{ine2}, we get
\begin{equation*}
	\begin{split}
	\Delta (W-W_{\bar R}) \geq & - \frac{4 \bar{\alpha} \Psi^2}{\left((1-\Psi^2)\xi - \bar{\alpha}\right)^2} \pscalar{\nabla (W-W_{\bar R})}{\nabla \xi}+ \frac{8 \bar{\alpha}^2 \Psi^2}{(1-\Psi^2) \left((1-\Psi^2) \xi - \bar{\alpha}\right)^2}W \\
	&- \frac{4 \bar{\alpha}^2 \Psi^4}{\left((1-\Psi^2) \xi - \bar{\alpha}\right)^4} W^2 - \frac{4 \xi \bar{\alpha}}{1-\Psi^2}+\frac{4 \Psi^2 \bar{\alpha}}{(1-\Psi^2) \left((1-\Psi^2) \xi - \bar{\alpha}\right)}W.
	\end{split}
\end{equation*}
Then, using that $$W_{\bar R}=\frac{\left((1-\Psi^2) \xi - \bar{\alpha}\right)^2 }{\Psi^2 (1-\Psi^2)} $$ and, after some simplifications, we get
\begin{equation}\label{ineLaplaciano}
	\begin{split}
	\Delta (W-W_{\bar R}) \geq & - \frac{4 \bar{\alpha} \Psi^2}{\left((1-\Psi^2)\xi - \bar{\alpha}\right)^2} \pscalar{\nabla (W-W_{\bar R})}{\nabla \xi} \\
	& + \frac{4 \Psi^2 \bar{\alpha}}{\left((1-\Psi^2) \xi -\bar{\alpha} \right)^2} \left(\xi - \frac{\Psi^2 \bar{\alpha}}{\left((1-\Psi^2) \xi -\bar{\alpha} \right)^2} \abs{\nabla \xi}^2 \right) (W-W_{\bar R}).
	\end{split}
\end{equation}
Hence, we have arrived at a differential inequality that is not elliptic at the quantity $W-W_{\bar R}$. 

To obtain an elliptic relation, we define the function $F_{\beta}=\beta \cdot (W-W_{\bar R})$, where $\beta=\beta(\Psi)>0$ is a function that we will determine later.
First, we observe that
\begin{equation}\label{lapF}
\Delta F_{\beta} = \beta \cdot \Delta (W-W_{\bar R}) + 2 \pscalar{\nabla \beta}{\nabla (W-W_{\bar{R}})} +   (W-W_{\bar R}) \Delta \beta .
\end{equation}
We compute the second and the third terms of the sum. Using that $\nabla \beta= \beta' (\Psi) \dot{\chi}(\xi) \nabla \xi$ (we use $\beta'$ to denote the derivative with respect to $\Psi$), we have that
\[
2 \pscalar{\nabla \beta}{\nabla (W-W_{\bar R})}= 2 \dot{\chi}\frac{\beta'}{\beta} \pscalar{\beta \nabla (W-W_{\bar R})}{\nabla \xi}.
\]
It is clear that
\[
\begin{split}
	\pscalar{\nabla F_{\beta}}{\nabla \xi} = \dot{\chi} \beta' W (W-W_{\bar R})+ \pscalar{\beta \nabla (W-W_{\bar R})}{\nabla \xi},
\end{split}
\]
so the second term of the sum is
\[
2 \pscalar{\nabla \beta}{\nabla (W-W_{\bar R})}= 2 \dot{\chi} \frac{\beta'}{\beta} \pscalar{ \nabla F_{\beta}}{\nabla \xi} - 2 \dot{\chi}^2 \left(\frac{\beta'}{\beta}\right)^2 W F_{\beta}.
\]
For the third term of the sum, using the chain rule we get 
\[
\Delta \beta = -2 \xi \dot{\chi} \beta'+\left(\ddot{\chi} \beta'+\dot{\chi}^2 \beta'' \right)W,
\]
so the third term of \eqref{lapF} is
\[
 (W-W_{\bar R}) \Delta \beta = -2 \xi \dot{\chi} \frac{\beta'}{\beta} F_{\beta}+ \left(\ddot{\chi} \frac{\beta'}{\beta}+ \dot{\chi}^2 \frac{\beta''}{\beta} \right)W F_{\beta}.
\]
Finally observe that from \eqref{primeraDerivadaChi} and \eqref{segundaDerivadaChi} we obtain, using that $\Psi (p)=\chi (\xi (p))$, that
\[
\dot{\chi}= \frac{\Psi (1-\Psi^2)}{(1-\Psi^2)\xi - \bar{\alpha}} \quad \textup{and} \quad \ddot{\chi} = \frac{2 \bar{\alpha} \Psi^3 (1-\Psi^2)}{\left((1-\Psi^2) \xi - \bar{\alpha}\right)^3},
\]
so combining \eqref{ineLaplaciano} and \eqref{lapF}, after some simplifications, it holds
\begin{equation}\label{ineqF}
	\begin{split}
		\Delta F_{\beta} \geq & \frac{2 \Psi (1-\Psi^2)}{(1-\Psi^2) \xi - \bar{\alpha}} \left(\frac{\beta'}{\beta}-\frac{2 \bar{\alpha} \Psi}{(1-\Psi^2) \left((1-\Psi^2)\xi - \bar{\alpha}\right)}\right) \pscalar{\nabla F_{\beta}}{\nabla \xi} \\
		& -\frac{2 \xi \Psi (1-\Psi^2)}{(1-\Psi^2)\xi - \bar{\alpha}} \left(\frac{\beta'}{\beta}-\frac{2 \Psi \bar{\alpha}}{(1-\Psi^2) \left((1-\Psi^2)\xi - \bar{\alpha}\right)}\right) F_{\beta} \\
		&+ \frac{\Psi^2 (1-\Psi^2)^2 \abs{\nabla \xi}^2}{\left((1-\Psi^2) \xi - \bar{\alpha}\right)^2}\Bigg( \left(\frac{\beta'}{\beta}\right)' -\left(\frac{\beta'}{\beta}\right)^2+ \frac{6 \Psi \bar{\alpha}}{(1-\Psi^2)\left((1-\Psi^2)\xi- \bar{\alpha}\right)} \frac{\beta'}{\beta} \\
		&-\frac{4 \Psi^2 \bar{\alpha}^2}{(1-\Psi^2)^2\left((1-\Psi^2)\xi - \bar{\alpha}\right)^2} \Bigg)F_{\beta}.
	\end{split}
\end{equation}
Now we are ready to prove Theorem \ref{theor_Gradient}.
\begin{proof}[Proof of Theorem \ref{theor_Gradient}:]\label{Proof of Theorem 4.1}
We give the proof in two parts: first, we will prove that $F_{\beta}$ satisfies an elliptic inequality, hence $W \leq W_{\bar{R}}$ in $\mathcal U$ and $W \equiv W_{\bar R}$ in $\mathcal U$ if  $W(p)=W_{\bar R}(p)$ at one single point $p \in \mathcal U$ by the maximum principle. Second, we will show the rigidity statement in the latter case. 

Define $F_{\beta}=\beta \cdot (W-W_{\bar R})$, where $$\beta (\Psi):= \frac{\sqrt{1-\Psi^2}}{\sqrt{W_{\bar R}}}>0$$ is a solution to the differential equation
\[
\frac{\beta'}{\beta}-\frac{2 \Psi \bar{\alpha}}{(1-\Psi^2)\left((1-\Psi^2)\xi - \bar{\alpha}\right)}=0.
\]

Take $\epsilon>0$ small enough and define $\mathcal{U}_{\epsilon}$ as in \eqref{innerRegionEpsilon}, using \eqref{ineqF} we get that $F_{\beta}$ satisfies the elliptic inequality
\[
\Delta F_{\beta}- \frac{8 \bar{\alpha} \Psi^4 (1-\Psi^2) \xi}{\left((1-\Psi^2)\xi - \bar{\alpha}\right)^4}  F_\beta\geq 0 \text{ in } \mathcal{U}_{\epsilon}. 
\]

Hence, using Lemma \ref{lemmaWR} and the Reverse Lojasiewicz Inequality \cite[Theorem 2.2]{Chr}, it follows that
\[
\lim_{x \in \mathcal{U}, x \to \textup{Max}(\xi)} \frac{W}{\sqrt{W_{\bar{R}}}}=\lim_{x \in \mathcal{U}, x \to \textup{Max}(\xi)} \frac{\abs{\nabla \xi}^2}{2 \sqrt{\xi_{\textup{max}}}\sqrt{\xi_{\textup{max}}-\xi}}=0,
\]
so it is clear that
\[
\lim_{x \in \mathcal{U}, x \to \textup{Max}(\xi)} F_{\beta} (p) = \lim_{x \in \mathcal{U}, x \to \textup{Max}(\xi)} \sqrt{1-\Psi^2} \frac{W}{\sqrt{W_{\bar R}}}-\sqrt{1-\Psi^2} \sqrt{W_{\bar R}}=0.
\]

Therefore, applying the maximum principle to $F_{\beta}$ in $\mathcal{U}_{\epsilon}$ and letting $\epsilon$ tend to zero (such that $\xi_{\textup{max}}-\epsilon$ is always a regular value) we conclude the first part of the proof.

We continue with the second part of the proof. We proceed as in \cite[Theorem 4.2]{Borg}. If we assume that $W(p)=W_{\bar R} (\Psi (\xi (p)))$ in $\mathcal U$, then it follows that $W= \abs{\nabla \xi}^2$ is a positive function in $\mathcal U$ that only depends on $\xi$ and hence it is constant along level sets. Therefore, there are no critical points of $\xi$ in $\mathcal U$, so we can parametrize $\mathcal U$ by level sets with coordinates $(\xi, \theta)$. In these coordinates, the metric on the sphere has the form
\[
g_{\sp^2}=\frac{1}{W} d \xi^2+ G(\xi, \theta) d \theta^2,
\]
where $G$ is a positive function. It is easy to compute the hessian of $\xi$ on level sets in these coordinates:
\[
\nabla_{i j}^2 \xi = \partial_{i j}^2 \xi - \Gamma_{i j}^k \partial_k \xi =-\Gamma_{i j}^{\xi},
\]
and
\[
-\Gamma_{\xi \xi }^{\xi} = \frac{\dot{W}}{2 W}, \, -\Gamma_{\xi \theta}^{\xi}=-\Gamma_{\theta \xi}^{\xi}=0, \, \text{and} \, -\Gamma_{\theta \theta}^{\xi}=- \frac{1}{2} W \frac{\partial G}{\partial \xi},
\]
 where $\dot{W}$ denote the derivative of $W$ with respect to $\xi$. Then we can compute the geodesic curvature of the level sets  with respect to the normal vector given by $\nabla \xi / \abs{\nabla \xi}$ easily. Using the formula given in \cite{Gold}, we have
\begin{equation}\label{GCurvature}
\kappa(\xi, \theta)= \frac{\nabla^2 \xi (\nabla \xi, \nabla \xi)- \abs{\nabla \xi}^2 \Delta \xi}{\abs{\nabla \xi}^3} = -\frac{\dot{W}}{2 W^{5/2}}+\frac{2 \xi}{W^{1/2}}.
\end{equation}

Thus, the geodesic curvature of the level sets only depends on $\xi$, so they are curves of constant geodesic curvature. As the level set curves do not cross, we have that they are intersections of parallel planes with $\mathbb{S}^2$. 

Then, it follows that $\Gamma=\partial \Omega \cap \overline{\mathcal U}$ is a circle on the sphere and $\xi= \bar \xi$ in a neighborhood of $\Gamma$ in $\mathcal U$, so we conclude that $\xi=\bar \xi$ in $\Omega$ since they are analytic functions. This concludes the proof of the theorem.

\end{proof}

\begin{remark}
	Note that we are using the sign convention for the curvature such that the geodesic curvature of $\partial \d ({\bf n}, s)$, where $\d ({\bf n}, s)$ is the geodesic disk in $\s^2$ centered at ${\bf n}$ of radius $s < \pi /2$, with respect to the inner normal is positive. It is important to remark on this because in \cite{ABM,Chr} the authors used the opposite sign convention.
\end{remark}

\subsection{Curvature estimates}

Here, we are going to obtain curvature estimates of the level sets of our solution $\xi$, taking advantage of the gradient estimates proved in Theorem \ref{theor_Gradient}. We recall that our model solution $\bar \xi$ satisfies:
\[
\nabla  \bar{\xi} = \frac{\sqrt{1-r^2}}{r} \left( \bar{\xi}- \frac{\bar{\alpha}}{1-r^2} \right)n
\]
and
\begin{equation*}
\nabla^2  \bar \xi=	\begin{pmatrix}
		- \bar \xi-\frac{\bar{\alpha}}{1-r^2} & 0\\
		0 & -\bar \xi + \frac{\bar{\alpha}}{1-r^2}
	\end{pmatrix},
\end{equation*}
so using \eqref{GCurvature}, we have that
\[
 \kappa(r) = \begin{cases} -\frac{r}{\sqrt{1-r^2}}, & \mbox{if } r \geq 0 ,\\ 
 \frac{r}{\sqrt{1-r^2}}, & \mbox{if } r <0 , 
 \end{cases}
\]
where the geodesic curvature is computed with respect to the inner orientation to $\mathcal{U}$, i.e., using the conormal vector given by $\nu=\nabla \xi / \abs{\nabla \xi}$.
\begin{figure}[htb]
	\centering
	\includegraphics[width=0.5\textwidth]{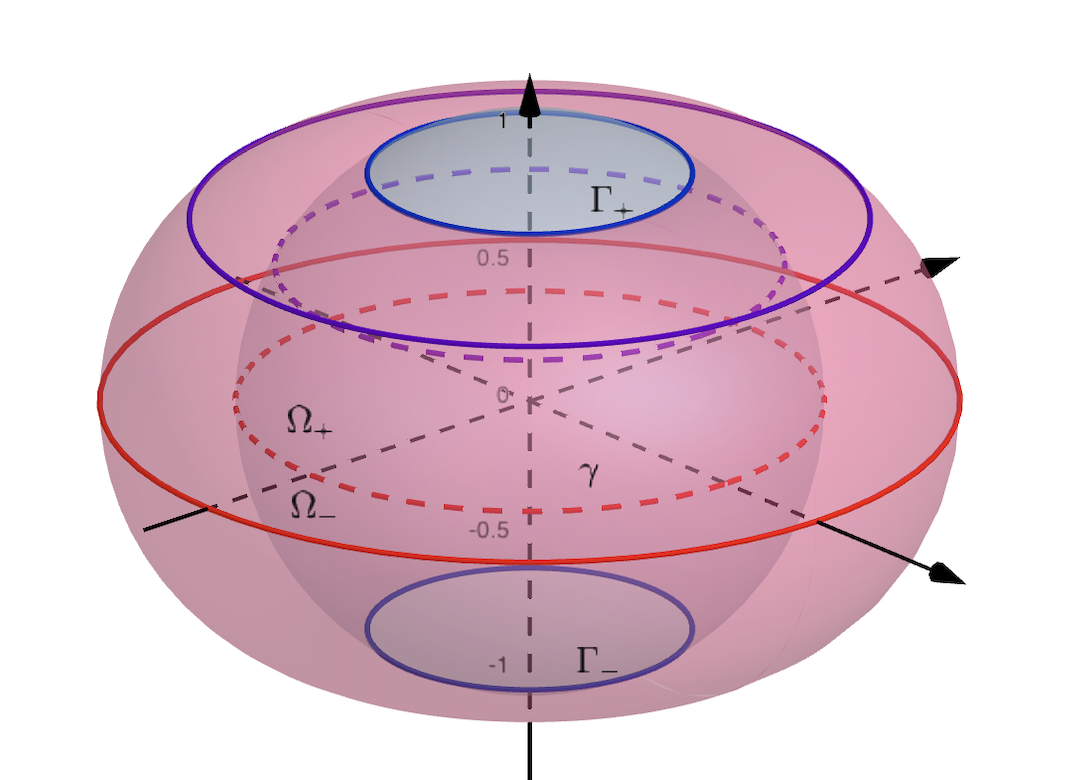}
	\caption{Here it appears the radial graph of the support function of the critical catenoid, which is the graph $\Sigma_{\xi_0}= \left\{(1+\xi_0 (r))(\sqrt{1-r^2} \cos \theta , \sqrt{1-r^2} \sin \theta , r) ~\colon~ r \in [-\bar r, \bar r], \theta \in [0, 2 \pi) \right\}$, where $\xi_0$ is the model solution defined in Section \ref{sec_2}. We see the zero-level sets in blue, the top-level set in red, and also the level set of height $r=0.5$ in purple. Here, the curvature of the level sets which are not in the top stratum is computed with respect to the conormal vector $\nu=\nabla \xi / \abs{\nabla \xi}$, which points to the curve $\gamma$ in each point of the domain.}
	\label{fig:graph}
\end{figure} 

In the next results, we estimate the geodesic curvature of the zero and top level sets of $\xi$.
\begin{proposition}\label{prop_curvatureZeroSets}
Let $(\Omega, \xi)$ be a solution to \eqref{DPS}, $\mathcal{U} \subset \Omega \setminus  \textup{Max} (\xi)$ a connected component and $\bar R= R(\mathcal{U})$ the expected critical height of the region. Let $(\bar \Omega, \bar \xi)$ be its associated model solution in $\mathcal{U}$ and $p \in \partial \Omega$ such that
\begin{equation}\label{condNabla}
\abs{\nabla \xi}^2 (p) = \max_{\partial \Omega \cap \textup{cl}(\U)} \abs{\nabla \xi}^2.
\end{equation}
Then, if $\kappa (p)$ denotes the curvature of $\partial \Omega$ at $p$ with respect the inner orientation to $\mathcal{U}$, it holds
\begin{equation*}
	\kappa (p) \leq -\frac{\bar{r}_+}{\sqrt{1-\bar{r}_+^2}} \quad \textup{if } \overline{\tau} (\mathcal{U}) \geq \tau_0 \quad \textup{and} \quad \kappa (p) \leq \frac{\bar{r}_{-}}{\sqrt{1-\bar{r}_{-}^2}} \quad \textup{if } \overline{\tau} (\mathcal{U}) < \tau_0,
\end{equation*}
where $\bar{r}_+ >0$ and $\bar{r}_- < 0$ are defined in \textup{Remark \ref{remarkNotation}}. 
\end{proposition}
\begin{proof}
Since the geodesic curvature of the level sets of $\xi$ does not change after a change of scale, we will assume that $\xi$ is normalized in the sense of Definition \ref{normalizedSolution}. Consider the curve 
\[
\sigma (s) = \text{exp}_p \left( \frac{\nabla \xi}{\abs{\nabla \xi}}(p) s \right), \quad s \in \mathbb{R},
\]
at some $p \in \textup{cl}(\U) \cap \partial \Omega$, here the exponential map is the one of $\s^2$. We compute the Taylor expansion of $W$ along this curve:
\[
W(\gamma (s))=W(p)+ \pscalar{\nabla W}{\nabla \xi / \abs{\nabla \xi}} (p) s + O (s^2), \quad s \in (-\epsilon, \epsilon),
\]
for some $\epsilon >0$ small enough. Now, it is easy to check that
\[
\kappa (p)=\frac{\nabla^2 \xi (\nabla \xi, \nabla \xi)- \abs{\nabla \xi}^2 \Delta \xi}{\abs{\nabla \xi}^3} = \frac{\pscalar{\nabla \abs{\nabla \xi}^2}{\nabla \xi}}{2 \abs{\nabla \xi}^3},
\]
so we have that
\[
W(\gamma (s))= W(p) + 2 \kappa(p) W(p) s + O (s^2), \quad s \in (-\epsilon, \epsilon).
\]
Next, we compute the Taylor expansion of $W_{\bar{R}}$ along the same curve:
\begin{eqnarray*}
	W_{\bar R} (\gamma (s)) &=& W_{\bar R} (p) + \pscalar{\nabla W_{\bar{R}} (p)}{\nabla \xi/ \abs{\nabla \xi} (p)} s + O(s^2) \\
	&=& W_{\bar{R}} (p)+ \left( \pm \frac{2 \Psi}{\sqrt{1-\Psi^2}}\sqrt{W_{\bar{R}}} -\frac{4 \bar{\alpha}}{1-\Psi^2} \right) \sqrt{W}(p) s + O(s^2),
\end{eqnarray*}
where $s \in [0, \delta)$ for some $0<\delta \leq \epsilon$. Since $W_{\bar{R}}(p)=W(p)$ at $p \in \partial \Omega $ by \eqref{condNabla} (this is clear using the definition of $\overline{\tau}(\mathcal{U})$), Theorem \ref{theor_Gradient} and the expansions of $W$ and $W_{\bar{R}}$ near $p$ imply
\[
\kappa  \leq \pm \frac{ \Psi}{\sqrt{1-\Psi^2}} \frac{\sqrt{W_{\bar{R}}}}{\sqrt{W}}  -\frac{2 \bar{\alpha}}{(1-\Psi^2) \sqrt{W}} \text{ at } p,
\]
where the sign is positive if $\overline{\tau}(\mathcal{U}) \geq \tau_0$ and negative if $\overline{\tau}(\mathcal{U}) < \tau_0$. Finally, since 
$$\sqrt{W}(p)=\sqrt{W_{\bar{R}}} (p) = \pm \frac{\bar{\alpha}}{\bar{r}_{\pm} \sqrt{1-  \bar{r}_{\pm}^2}} \text{ and }\Psi (p) = \bar{r}_{\pm},$$we conclude the following estimate for the curvature of the zero level set of $\mathcal{U}$:
\[
\kappa (p) \leq \begin{cases} -\frac{\bar{r}_{+}}{\sqrt{1-\bar{r}_{+}^2}}, & \mbox{if } \overline{\tau}(\mathcal{U}) \geq \tau_0 \\ \frac{\bar{r}_{-}}{\sqrt{1-\bar{r}_{-}^2}}, & \mbox{if } \overline{\tau}(\mathcal{U}) < \tau_0. \end{cases}
\]
\end{proof}

\begin{proposition}\label{prop_curvatureMaxSet}
Let $(\Omega, \xi)$ be a solution to \eqref{DPS}, $\mathcal{U} \subset \Omega \setminus  \textup{Max} (\xi)$ a connected component and $\bar R= R(\mathcal{U})$ the expected critical height of the region. Let $(\bar \Omega, \bar \xi)$ be its associated model solution in $\mathcal{U}$. Let $\gamma \subset \textup{cl}(\U) \cap \textup{Max}(\xi)$ be an analytic curve and $p \in \gamma$. Then if $\kappa (p)$ denotes the curvature of $\gamma$ at $p$ with respect the inner orientation to $\mathcal{U}$, it holds
\begin{equation*}
	\kappa (p) \leq \frac{\bar R}{\sqrt{1-\bar{R}^2}} \quad \textup{if } \overline{\tau} (\mathcal{U}) \geq \tau_0 \quad \textup{and} \quad \kappa (p)\leq -\frac{\bar{R}}{\sqrt{1-\bar{R}^2}} \quad \textup{if } \overline{\tau} (\mathcal{U}) < \tau_0.
\end{equation*}
\end{proposition}
\begin{proof}
As in the proof of Proposition \ref{prop_curvatureZeroSets}, we assume again that the solution $\xi$ is normalized. To estimate the geodesic curvature of $\gamma$ we use \cite[Theorem 3.1]{Chr}. If we denote $$\rho (x)= \textup{dist}(x,\gamma), \quad x \in \mathcal{U},$$ then
\[
\xi=\xi_{\textup{max}} - \frac{\varphi}{2}\rho^2-\frac{\varphi (\xi_{\textup{max}})}{6} \kappa \, \rho^3+ O(\rho^4),
\]
where $\varphi$ is such that $\Delta \xi= -\varphi(\xi)$. Replacing the explicit terms in the previous expansion, we obtain:
\[
\xi=\frac{\bar{\alpha}}{1-\bar{R}^2}-\frac{\bar{\alpha}}{1-\bar{R}^2}\rho^2-\frac{\bar{\alpha}}{3(1-\bar{R}^2)} \kappa \, \rho^3+ O(\rho^4)
\]
in a certain neighborhood of $\gamma$ inside $\U$. Computing the gradient of the previous expression, we get
\[
\nabla \xi= -\frac{\bar{\alpha}}{1-\bar{R}^2} \rho \left(2+ \kappa \, \rho \right)\partial_\rho+O(\rho^3),
\]
and then we conclude that
\begin{equation}\label{TaylorW}
	W= \frac{4\bar{\alpha}^2 \rho^2}{(1-\bar{R}^2)^2}\left(1+\kappa \, \rho\right)+O(\rho^4).
\end{equation}

We want to obtain an expansion of $W_{\bar{R}}$ as the previous one. To achieve this, we use the Taylor expansions obtained in the proof of Lemma \ref{lemmaWR}. Set $z=\Psi-\bar{R}$. Observe that from \eqref{expansionXi} it is easy to check that
\[
\frac{\xi_{\textup{max}}-\xi}{z^2}=\frac{\xi_{\textup{max}}-\xi}{(\Psi-\bar R)^2} \rightarrow \textup{constant, as }  \Psi \rightarrow \bar{R}, 
\]
so we conclude that
\begin{equation}\label{expZ}
	z=\pm \frac{1-\bar{R}^2}{\sqrt{\bar{\alpha}}}\sqrt{\xi_{\textup{max}}-\xi}+ O \left(\xi_{\textup{max}}-\xi\right)^{3/2},
\end{equation}
where the sign of the first term is positive if $\overline{\tau}(\U)\geq \tau_0$ and negative otherwise.

We will need the third order term in the expansion \eqref{expansionWR}, so we compute the third derivative of $W_{\bar{R}}$. A straightforward computation shows
\[
\frac{\partial^3 W_{\bar{R}}}{\partial \Psi^3} (\bar{R})=-12 \bar{R} \frac{4 \bar{\alpha}^2}{(1-\bar{R}^2)^4}+6 (1-\bar{R}^2)\left(\frac{2 \bar{\alpha}}{(1-\bar{R}^2)^2} \frac{8 \bar{\alpha} \bar{R}}{(1-\bar{R}^2)^3}\right)=\frac{48 \bar{R} \bar{\alpha}^2}{(1-\bar{R}^2)^4},
\]
so the expansion of $W_{\bar{R}} (\Psi)$ at $\Psi=\bar{R}$ up to third order is given by
\[
W_{\bar{R}}= \frac{4 \bar{\alpha}^2}{(1-\bar{R}^2)^3}z^2+\frac{8\bar{R} \bar{\alpha}^2}{(1-\bar{R}^2)^4}z^3+O(z^4).
\]
Substituting \eqref{expZ} in the previous expansion we get
\begin{equation*}
	W_{\bar{R}} =\frac{4 \bar{\alpha}}{1-\bar{R}^2}(\xi_{\textup{max}}-\xi) \pm \frac{8 \bar{R}\bar{\alpha}^{1/2}}{1-\bar{R}^2}(\xi_{\textup{max}}-\xi)^{3/2}+O(\xi_{\textup{max}}-\xi)^2,
\end{equation*}
using again \cite[Theorem 3.1]{Chr} and simplifying we conclude that
\begin{equation}\label{TaylorWR}
	W_{\bar{R}} =  \frac{4 \bar{\alpha}^2}{(1-\bar{R}^2)^2} \rho^2 \left(1+\left(\frac{1}{3} \kappa  \pm \frac{2}{3} \frac{\bar{R}}{\sqrt{1-\bar{R}^2}} \right)\rho \right)+ O(\rho^{7/2}).
\end{equation}

Finally, we proceed as at the end of the proof of Proposition \ref{prop_curvatureZeroSets}. By Theorem \ref{theor_Gradient} we have that $W \leq W_{\bar{R}}$ in $\mathcal{U}$, comparing the expansions near $\gamma$,  \eqref{TaylorW} and \eqref{TaylorWR}, we have that
\[
 \kappa (p) \leq \frac{\kappa (p)}{3} \pm \frac{2}{3} \frac{\bar{R}}{\sqrt{1-\bar{R}^2}} ,
\]that is,
\[
\kappa (p) \leq \begin{cases} \frac{\bar{R}}{\sqrt{1-\bar{R}^2}}, & \mbox{if } \overline{\tau}(\mathcal{U}) \geq \tau_0 \\ -\frac{\bar{R}}{\sqrt{1-\bar{R}^2}}, & \mbox{if } \overline{\tau}(\mathcal{U}) < \tau_0. \end{cases}
\]
\end{proof}

At this point, we can announce the following result:
\begin{theorem}\label{curvature_Theorem}
Let $(\Omega, \xi)$ be a solution to \eqref{DPS}. Assume that there exists an analytic closed curve $\gamma \subset \textup{Max}(\xi)$, let $\Omega^\gamma_1, \Omega^\gamma_2 \subset \Omega$ be the partition with respect to $\gamma$ and set $R_{1}=R(\Omega^\gamma_{1})$ and $R_{2}=R(\Omega^\gamma_{2})$ the expected critical heights associated to that regions.
\begin{itemize}
\item If $\overline{\tau}(\Omega^\gamma_{1}) \leq \tau_0$, then $\overline{\tau}(\Omega^\gamma_{2}) \geq \tau_0$ and for each $p \in \gamma$ it holds
\begin{equation}\label{curvatureTheorem}
	\frac{R_1}{\sqrt{1-R_{1}^2}} \leq \kappa (p) \leq \frac{R_2}{\sqrt{1-R_{2}^2}},
\end{equation}
where $\kappa(p)$ is the geodesic curvature of $\gamma$ in $p$ with respect to the normal vector pointing to the region $\Omega^\gamma_{2}$.
\item If $\overline{\tau} (\Omega^\gamma_{1})=\overline{\tau} (\Omega^\gamma_{2})=\tau_0$, then $\kappa = 0$ along $\gamma$.
\end{itemize}

In particular, $R_2 \geq R_1$ and equality holds if, and only if, $(\Omega, \xi) \equiv (\bar \Omega, \bar \xi)$, where $(\bar \Omega, \bar \xi)$ is the associated model solution given in Definition \ref{associatedSolution}.
\end{theorem}
\begin{proof}
 First, since we are assuming that $\overline{\tau} (\Omega^\gamma_{1})\leq\tau_0$, by Proposition \ref{prop_curvatureMaxSet} we conclude that
\begin{equation*}
	0 \leq \frac{R_1}{\sqrt{1-R_1^2}} \leq \kappa (p) \quad \forall p \in \gamma,
\end{equation*}
so, in particular, $\kappa \geq 0$ with respect to the normal orientation pointing to $\Omega^\gamma_{2}$. Then $\overline \tau (\Omega^\gamma_{2}) \geq \tau_0$, because otherwise, using Proposition \ref{prop_curvatureMaxSet} with this region, we would conclude that $\kappa < 0$ along $\gamma$, a contradiction. It follows also that it must be $R_1 \geq R_2$, and thus we obtain the chain of inequalities \eqref{curvatureTheorem}. 

Assume now that $R_1=R_2=\bar R$, then $\gamma$ must have constant geodesic curvature, so it must be a parallel at height $\bar R$. Let us suppose that $\xi$ is normalized. Then, we get that $(\Omega, \xi)$ is a model solution as follows: take $\mathcal{N}$ a tubular neighborhood of $\gamma$ of radius $\epsilon< \min{(-\bar r_-, \bar r_+)}$, then $\xi$ is a solution to 
\[
\left\{	\begin{matrix}
	\Delta \xi + 2 \xi=0 & \text{ in } & \mathcal{N} ,\\[2mm]
	\xi =  \xi_{\textup{max}}& \text{ along } & \gamma  ,\\[2mm]
	\frac{\partial \xi}{ \partial \nu}= 0 & \text{ along } & \gamma.
\end{matrix}\right.
\]
Since $\gamma$ is non-characteristic, it follows from the Cauchy-Kovalevskaya theorem (cf. \cite{Kr}) that $\xi$ is the unique solution to the previous problem. But the associated model solution $(\bar \Omega, \bar \xi)$ is also a solution to the previous problem, so we conclude that $\xi \equiv \bar \xi$ in $\mathcal{N}$ (up to a rotation if necessary). It is clear then that it must be $(\Omega, \xi) \equiv (\bar \Omega, \bar \xi)$ by analyticity.

Finally, consider the case $\overline{\tau} (\Omega^\gamma_{1})=\overline{\tau} (\Omega^\gamma_{2})=\tau_0$. It follows from Proposition \ref{prop_curvatureMaxSet} that $\kappa \leq 0$ with respect to the inner and to the outer normal to $\Omega^\gamma_{2}$ simultaneously, so it must be $\kappa =0$ along $\gamma$. Then we get that it must be $(\Omega, \xi) \equiv (\Omega_0, \xi_0)$ using the same argument as in the previous case.
\end{proof}

\subsection{Length estimates}

We show here a relation between the length of the zero and top level sets: 
\begin{proposition}\label{lengthBounds}
	Let $(\Omega, \xi)$ be a solution to \eqref{DPS}, $\mathcal{U} \subset \Omega \setminus  \textup{Max} (\xi)$ a connected component and $\bar R= R(\mathcal{U})$ the expected critical height of the region. Suppose that $\textup{cl}(\U) \cap \textup{Max}(\xi)= \gamma^{\U}$ and $\textup{cl}(\U) \cap \partial \Omega = \Gamma^{\U}$ are sets of analytic closed curves. Then 
	\begin{equation}\label{lengthBounds}
		\frac{\abs{\gamma^{\U}}}{\sqrt{1- \bar{R}^2}} \leq \frac{\abs{\Gamma^{\U}}}{\sqrt{1-r_+^2} } \quad \textup{if } \overline{\tau} (\mathcal{U}) \geq \tau_0 \quad \textup{and} \quad 		\frac{\abs{\gamma^{\U}}}{\sqrt{1- \bar{R}^2}} \leq \frac{\abs{\Gamma^{\U}}}{\sqrt{1-r_-^2} } \quad \textup{if } \overline{\tau} (\mathcal{U}) < \tau_0,
	\end{equation}
	where $\abs{\gamma^{\U}}$ and $\abs{\Gamma^{\U}}$ denotes the sum of the lengths of each set of curves. 
\end{proposition}
\begin{proof}
	We follow the proof of \cite[Proposition 5.5]{ABM}. Assume that $\xi$ is normalized in the sense of Definition \ref{normalizedSolution}. Given $\epsilon>0$, set $\U_{\epsilon}$ as in \eqref{innerRegionEpsilon} (as always, we suppose that $\xi_{\textup{max}}-\epsilon$ is regular). Then, by the divergence theorem, it holds:
	\begin{equation*}
		\begin{split}
			\int_{\U_{\epsilon}} \frac{-2 \xi}{(1-\Psi^2) \xi- \bar \alpha} \, dA &= \int_{\U_{\epsilon}} \frac{\Delta \xi}{(1-\Psi^2) \xi- \bar \alpha} \, dA \\
			&= \int_{\U_{\epsilon}} \pscalar{\frac{\nabla \xi}{(1-\Psi^2)\xi- \bar \alpha}}{\nabla\xi} \, dA+\int_{\partial \U_{\epsilon}} \frac{\pscalar{\nabla\xi}{\nu}}{(1-\Psi^2) \xi- \bar \alpha} \, ds,
		\end{split}
	\end{equation*}
	where $\Psi$ is the pseudo-radial function given in Definition \ref{psudoRadialFunctions}, $\bar \alpha=\alpha (\bar R)$ as in Proposition \ref{prop_2.1} and $\nu$ is the inner unit normal to $\U_{\epsilon}$. Thus, on the one hand 
	\[
	\nabla \frac{\Psi}{(1-\Psi^2) \xi- \bar \alpha}= \frac{2 \Psi^3 (1-\Psi^2) \xi}{\left((1-\Psi^2)\xi - \bar \alpha\right)^3} \nabla \xi
	\]
	and on the other hand
	\[
	\frac{\left((1-\Psi^2)\xi - \bar \alpha\right)^3}{2 \Psi^3 (1-\Psi^2) \xi} \cdot \frac{2 \Psi \xi}{(1-\Psi^2)\xi-\bar \alpha} = \frac{\left((1-\Psi^2)\xi - \bar \alpha\right)^2}{\Psi^2 (1-\Psi^2)}= W_{\bar{R}},
	\]
	so we obtain the identity
	\begin{equation}\label{0}
		\int_{\U_{\epsilon}} \frac{2 \Psi^3 (1-\Psi^2) \xi}{\left((1-\Psi^2)\xi - \bar \alpha\right)^3} (W-W_{\bar{R}}) \, dA = -\int_{\Gamma_{\U}} \frac{\abs{\nabla \xi} \Psi}{(1-\Psi^2)\xi - \bar \alpha} \, ds + \int_{\gamma^{\U}_{\epsilon}} \frac{\abs{\nabla \xi} \Psi}{(1-\Psi^2)\xi - \bar \alpha} \, ds,
	\end{equation}
	where $\gamma^{\U}_{\epsilon}=\{p \in \U ~\colon~ \xi(p)+\epsilon=\xi_{\textup{max}} \}$. It is clear that for $\epsilon>0$ small enough,  $\gamma^{\U}_{\epsilon}$ is a set of analytic curves, and $\gamma_{\epsilon}^\U \to \gamma^\U$ when $\epsilon \to 0$.
	
Next, we analyze both sides of the previous identity. We begin with the right-hand side. Since $\Psi=r_{\pm}$ and $\max_{\Gamma^{\U}} \abs{\nabla \xi} = \pm \bar \alpha / r_{\pm} \sqrt{1-r_{\pm}^2}$, it is clear that
	\begin{equation}\label{1}
		\int_{\Gamma^{\U}} \frac{\abs{\nabla \xi }\Psi}{(1-\Psi^2)\xi- \bar \alpha} \, ds= \mp \frac{1}{\sqrt{1-r_{\pm}^2}} \int_{\Gamma^{\U}} \frac{\abs{\nabla \xi}}{\max_{\Gamma^{\U}} \abs{\nabla \xi}}
		\left\{	\begin{matrix}
			\geq - \frac{\abs{\Gamma^{\U}}}{\sqrt{1-r_+^2}} & \text{ if } & \overline{\tau} (\mathcal{U}) \geq \tau_0 ,\\[2mm]
			\leq \frac{\abs{\Gamma^{\U}}}{\sqrt{1-r_-^2}} & \text{ if } & \overline{\tau} (\mathcal{U}) < \tau_0.
		\end{matrix}\right.
	\end{equation}
Also, note that by the Taylor expansions given in \eqref{TaylorW} and \eqref{TaylorWR} we have
\[
\lim_{p \to \gamma^{\U}} \frac{W}{W_{\bar{R}}} = 1.
\]  
Since
\[
\frac{\abs{\nabla \xi }\Psi}{(1-\Psi^2)\xi- \bar \alpha} = \mp \frac{1}{\sqrt{1-\Psi^2}} \sqrt{\frac{W}{W_{\bar{R}}}},
\]
(as always, the upper sign corresponds to the case $\overline{\tau}(\U) \geq \tau_0$ and the lower sign otherwise) we conclude that
\begin{equation}\label{2}
 \lim_{\epsilon \to 0^+} \int_{\gamma^{\U}_{\epsilon}} \frac{\abs{\nabla \xi} \Psi}{(1-\Psi^2)\xi - \bar \alpha} \, ds = \mp \frac{\abs{\gamma^{\U}}}{1-\bar{R}^2}.
\end{equation}

Now, we look at the left-hand side of \eqref{0}. Note that 
\[
\frac{\Psi^3 (1-\Psi^2) \xi}{\left((1-\Psi^2)\xi - \bar \alpha\right)^3}= \mp \frac{1}{ W_{\bar R} ^{3/2}  \sqrt{1-\Psi^2}},
\]
therefore, using Theorem \ref{theor_Gradient}, we conclude that
\begin{equation}\label{3}
\int_{\U_{\epsilon}} \frac{2 \Psi^3 (1-\Psi^2) \xi}{\left((1-\Psi^2)\xi - \bar \alpha\right)^3} (W-W_{\bar{R}}) \, dA
\left\{	\begin{matrix}
\geq 0& \text{ if } & \overline{\tau} (\mathcal{U}) \geq \tau_0 ,\\[2mm]
	\leq 0 & \text{ if } & \overline{\tau} (\mathcal{U}) < \tau_0.
\end{matrix}\right.
\end{equation}

Finally, combining \eqref{1}, \eqref{2} and \eqref{3} we obtain the inequalities \eqref{lengthBounds}.
\end{proof}


\section{Overdetermined solutions}\label{sec_5}

Let $\Omega \subset \s ^2$ be a domain of finite type with $k \geq 2$ boundary components and consider the overdetermined elliptic problem
\begin{equation}\label{OEP}
	\left\{	\begin{matrix}
		\Delta \xi + 2 \xi=0 & \text{ in } & \Omega ,\\[2mm]
		\xi >0 & \text{ in } & \Omega  ,\\[2mm]
		\xi = 0 & \text{ along } & \partial \Omega, \\[2mm]
		\abs{\nabla \xi} = b_i >0 &  \text{ along } & \Gamma_i, & \text{ for }  i \in \{1, \dots, k\},
	\end{matrix}\right.
\end{equation}
where we have that $\partial \Omega= \bigcup_{i=1}^k \Gamma_{i}$ using the notation introduced in Definition \ref{finiteTypeDomain}.
\begin{remark}
When $\Omega$ is a domain of finite type with $k =1$, i.e., a topological disk, a solution $\xi $ to \eqref{OEP} must be rotationally symmetric and $\Omega$ a geodesic disk by \cite{EM,Sou}, that is, $(\xi , \Omega) \equiv (\xi _1 , \Omega _1)$.
\end{remark}
In this section, we will classify the solutions $(\Omega, \xi)$ to \eqref{OEP} which have infinitely many maximum points. To do so, we will use the estimates computed in Section \ref{sec_4} and a length estimate that uses the overdetermined condition and Proposition \ref{prop_curvatureZeroSets}.
\begin{proposition}\label{lengthZeroLevelSets}
Let $(\Omega, \xi)$ be a solution to \eqref{OEP} and $\mathcal{U} \subset \Omega \setminus  \textup{Max} (\xi)$ an annular connected component, i.e., homeomorphic to an annulus, such that $\partial \mathcal U = \Gamma \cup \gamma$, $\Gamma \subset \partial \Omega$ and $\gamma \subset {\rm Max}(\xi)$. Then, it holds
\begin{equation*}
	\abs{\Gamma} \leq 2 \pi \sqrt{1-r_+^2} \quad \textup{if } \overline{\tau} (\mathcal{U}) \geq \tau_0 \quad \textup{and} \quad 	\abs{\Gamma} \leq 2 \pi \sqrt{1-r_-^2}\quad \textup{if } \overline{\tau} (\mathcal{U}) < \tau_0.
\end{equation*}
\end{proposition}
\begin{proof}
First, observe that Proposition \ref{prop_curvatureZeroSets} and the overdetermined condition $\xi =0$ and $\abs{\nabla \xi} = b_i >0 $ along each $\Gamma _i \subset \partial \Omega$ connected component imply that 
$$\abs{\kappa (p)} \geq \abs{\frac{\bar r_\pm}{\sqrt{1-\bar r_\pm}}} \text{ for all } p \in \Gamma $$where $\bar{r}_+ >0$ and $\bar{r}_- < 0$ are defined in \textup{Remark \ref{remarkNotation}} and $\kappa$ is the geodesic curvature of $\Gamma$. Then, by Blaschke turning theorem (generalized to space forms in \cite{Kar}) we conclude that there exists $p \in \mathbb{S}^2$ such that $\Gamma$ is contained in the geodesic disk  $D:=\mathbb{D}(p, |r_{\pm}|)$.  Now, given a closed convex curve $\sigma : I \to \mathbb{S}^2$, Crofton's formula (cf. \cite{montielRos}) tells us that
\[
\textup{Length}(\sigma)= \frac{1}{4} \int_{\mathbb{S}^2} n_{\sigma} (a) \, da,
\]
where $n_{\sigma} (a)$ measures the number of points in which the plane $P_a := <\set{a}>^\perp $, $a \in \mathbb{S}^2$, and the curve $\sigma$ intersects. Since any $P_a$ intersects an embedded strictly convex closed curve in the sphere at most at two points and $\Gamma \subset D$, the result follows. Recall that $\Gamma$ is strictly convex with the correct orientation.
\end{proof}

Now, we have the ingredients to prove our main classification result:

\begin{figure}[htb]
	\centering
	\includegraphics[width=0.5\textwidth]{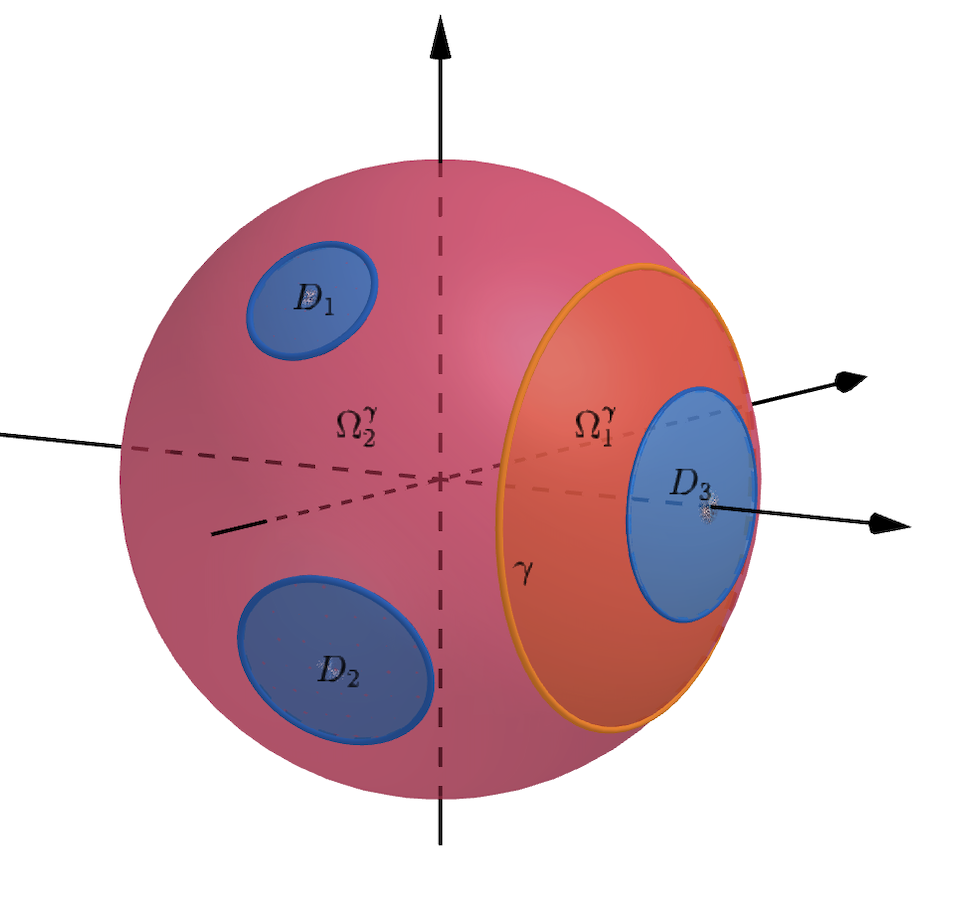}
	\caption{Here, a finite domain in the unit sphere. The domain is $\Omega=\mathbb{S}^2 \setminus \bigcup_{i=1}^3 D_i$, where $D_1, D_2$ and $D_3$ are the blue regions. We have that $\Omega \setminus \gamma= \Omega^\gamma_1 \cup \Omega^\gamma_2$ is a partition of $\Omega$, and the region $\Omega^\gamma_1$ has the topology of an annulus. If $\xi$ solves \eqref{OEP} and $\gamma \in \textup{Max}(\xi)$ then $(\Omega, \xi)$ satisfies the hypothesis of Theorem \ref{rigidityConstantNeumannOEP}.}
	\label{fig:finiteDomain}
\end{figure} 

\begin{theorem}\label{rigidityConstantNeumannOEP}
Let $(\Omega, \xi)$ be a solution to the overdetermined problem \eqref{OEP}, where $\Omega$ is a $\mathcal{C}^2$-domain of finite type with $k \geq 2$ boundary components. Suppose that $\textup{Max}(\xi)$ contains a closed curve $\gamma$ such that the partition with respect to $\gamma$, $\Omega \setminus \gamma = \Omega _1 ^\gamma \cup \Omega _2 ^\gamma$ contains a component, say $\Omega _1 ^\gamma$, that is a topological annulus. Then, $k=2$ and $\Omega$ is rotationally symmetric. In particular, there exists an $R \in [0,1)$ such that $(\Omega, \xi) = (\Omega_R, \xi_R)$ up to a rotation and a dilation.
\end{theorem}
\begin{proof}

By hypothesis,  $\Omega ^\gamma_1$ is a topological annulus and, up to rearranging the indexes if necessary, $\partial \Omega _1 ^\gamma = \Gamma _1 \cup \gamma$, where $\Gamma _1 = \partial \Omega \cap \partial \Omega _1 ^\gamma$. It follows from Proposition \ref{lengthBounds} and Proposition \ref{lengthZeroLevelSets} that

$$\abs{\gamma_{1}} \leq 2 \pi \sqrt{1-R_1^2} \frac{\abs{\Gamma_1}}{\sqrt{1-r_{\pm}^2}} \leq 2 \pi \sqrt{1-R_1^2},$$where $R_1 =R(\Omega _1 ^\gamma) \in [0,1)$ is the expected critical height of the region $\Omega _1 ^\gamma$. Then, \cite[Proposition 9.54]{montielRos} implies that $\gamma_1$ is either an equator or it is contained in an open hemisphere, denoted by $\s ^2 _+$. 

In the first case, we conclude that $(\Omega, \xi) = (\Omega_0, \xi_0)$ up to a rotation and a dilation applying the Cauchy-Kovalevskaya theorem. In the second case, since $\gamma $ is not contractible in $\Omega$ by Lemma \ref{lemmaCritic}, either $\Omega^\gamma _1 \subset \s ^2 _+$ or $\Omega _2 ^\gamma \subset \s^2 _+$. In any case, we can apply the moving plane method in $\s^2 _+$ (a slight modification of the arguments in \cite{KP,QXia}) to the component $\Omega _i ^\gamma \subset \s ^2 _+$ to conclude that $\Omega$ and $\xi$ are rotationally symmetric. Since the model solutions are the only solutions to \eqref{DPS} with this property, this proves the result.
\end{proof}

\begin{remark}\label{RemTop}
Let $\Omega $ be a finite domain with $k \geq 2$ and $\xi $ a solution to \eqref{OEP}, a simple topological argument and Lemma \ref{lemmaCritic} imply that, if the top level set of $\xi$ contains $k-1$ curves $\gamma _i$, $i=1, \ldots , k$, then $\Omega \setminus \bigcup _{i=1}^{k-1} \gamma _i$ contains an annular component in the conditions of Theorem \ref{rigidityConstantNeumannOEP}. Hence, $(\Omega, \xi) = (\Omega_R, \xi_R)$ up to a rotation and a dilation.
\end{remark}

In the case that $\Omega $ is a topological annulus, Theorem \ref{rigidityConstantNeumannOEP} and Remark \ref{RemTop} imply:
\begin{quote}
{\bf Theorem A:} {\it Let $(\Omega, \xi)$ be a solution to \eqref{OEP}, $\Omega$ an annular domain (i.e., $k=2$). Suppose that $\xi$ has infinitely many maximum points inside $\Omega$. Then, there exist $R \in [0,1)$ such that $(\Omega, \xi) \equiv  (\Omega_R, \xi_R)$.}
\end{quote}

We conclude this section by noting that we can change the topological condition on $\Omega \setminus \gamma$ in Theorem \ref{rigidityConstantNeumannOEP} by a restriction in the NWSS of the solution. In fact, using Theorem \ref{curvature_Theorem} we can obtain the following result:
\begin{theorem}\label{rigidityConstantNeumannOEP2}
	Let $(\Omega, \xi)$ be a solution to the overdetermined problem \eqref{OEP}, where $\Omega$ is of finite type with $k \geq 2$. Suppose that $\xi$ has infinitely many maximum points and that at least one of the connected components of $\Omega \setminus \textup{Max}(\xi)$ has NWSS less or equal to $\tau_0$. Then, $\Omega$ is rotationally symmetric. In particular, there exists $R \in [0,1)$ such that $(\Omega, \xi) \equiv (\Omega_R, \xi_R)$.
\end{theorem}

\begin{proof}
Let $\U \subset \Omega \setminus \textup{Max}(\xi)$ such that $\overline{\tau}(\U) \leq \tau_0$ by hypothesis and set $\partial \mathcal U \cap {\rm Max}(\xi) = \bigcup _{j=1}^m \gamma _j$, the boundary components of $\mathcal U $ in the top level set. Theorem \ref{curvature_Theorem} implies that the geodesic curvature of $\gamma_j$, with respect to the inner orientation to $\mathcal{U}$, is non-positive. Therefore, each $\gamma _j$ is contained in a closed hemisphere and hence, there exists, at least, a connected component $\gamma \subset \partial \mathcal U \cap {\rm Max}(\xi) $ such that one of the components, $\Omega \setminus \gamma = \Omega^\gamma _1 \cup \Omega _2 ^\gamma$, of the partition with respect to $\gamma$, say $\Omega _1 ^\gamma$, is contained in a closed hemisphere, denoted by $\textup{cl}(\mathbb{S}^2_+)$.

We distinguish two cases. First, assume that the geodesic curvature of $\gamma$ vanishes identically. In this case, Cauchy-Kovalevskaya theorem implies that $(\xi , \Omega) = (\xi _0 , \Omega _0)$, up to a rotation and scaling. Second, if the geodesic curvature does not vanish identically, we can move slightly $\textup{cl}(\mathbb{S}^2_+)$ to obtain an open hemisphere, still denoted by $\s ^2_+$, such that $\textup{cl}(\Omega _1^\gamma) \subset \s _2 ^+$. Hence, again by the moving plane method using the overdetermined condition, we infer that  $(\xi , \Omega) \equiv (\xi _R , \Omega _R)$ for some $R \in (0,1)$.
\end{proof}

\section{Geometric application}\label{sec_6}

In this section, we will exploit the correspondence of solutions to \eqref{OEP} and free boundary minimal surfaces explained in the introduction. Let $\mathbb{B}^3 (R)$ denote the open Euclidean ball of radius $R>0$ centered at the origin, $\partial \mathbb{B}^3 (R)= \mathbb{S}^2 (R)$. For simplicity, we also set $\mathbb{B}^3=\mathbb{B}^3 (1)$ and $\s ^2 \equiv \s ^2 (1)$.

\begin{definition}\label{surfacesWithFreeBoundaries}
	Let $\Sigma \subset \mathbb{R}^3$ be an open immersed minimal surface with boundary. We will say that $ \Sigma$ has \textup{free boundaries} if each boundary component of $ \Sigma$ meets orthogonally a sphere centered at the origin, possibly of different radii. If $\Sigma \subset \mathbb{B}^3$ and $\partial \Sigma \subset \mathbb{S}^2$ then $\Sigma$ is said to be a free boundary minimal surface in the unit ball.
\end{definition}

We recall the correspondence between minimal surfaces with free boundaries and the solutions to \eqref{OEP} (cf. \cite[Proposition 2.1]{Sou}).
\begin{proposition}\label{Correspondence}
	Let $\Omega \subset \mathbb{S}^2$ be a $\mathcal{C}^2$-domain in the sphere. Suppose that $\Omega$ is not a topological disk and there exists a nonzero function $\xi \in \mathcal{C}^2 (\Omega) \cap \mathcal{C}^1 (\textup{cl}(\Omega))$ solution to the problem
	\begin{equation}\label{OEPGeneral}
		\left\{	\begin{matrix}
			\Delta \xi + 2 \xi=0 & \text{ in } & \Omega ,\\[2mm]
			\xi = 0 & \text{ along } & \partial \Omega, \\[2mm]
			\abs{\nabla \xi}^2= b_{i}^2 & \text{ along } & \Gamma _i \in \pi_0 (\partial \Omega), \, i \in \set{1, \ldots ,k},
		\end{matrix}\right.
	\end{equation}
	where $\nu$ is the exterior unit normal to $\Omega$ and $b_{i}$ is a constant for each $i \in \set{1, \ldots ,k}$. Then, the map:
	\begin{equation*}
		\begin{split}
			X_{\xi}:  \, & \Omega \longrightarrow \mathbb{R}^3 \\
			& \, z \,\, \mapsto X_{\xi} (z):= \nabla^{\sp^2} \xi (z)+ \xi (z) \cdot z.
		\end{split}
	\end{equation*}
	defines a branched minimal surface in $\mathbb{R}^3$ with free boundaries, where each boundary component $X_{\xi} (\Gamma _i)$ lies in the sphere $\mathbb{S}^2 \left(\abs{b_i}\right)$.
	
	Conversely, let $\Sigma$ be a minimal surface with free boundaries, $\partial \Sigma = \bigcup _{i=1}^k \zeta _i$, and injective Gauss map $N: \Sigma \to \mathbb{S}^2$. Set $u(p) := \meta{p}{N(p)}$ its support function. Then
	\begin{equation}\label{supportFunction}
		\xi (z) := (u \circ N^{-1})(z) =\pscalar{N^{-1} (z)}{z}, \quad \forall z \in N (\Sigma)= \Omega
	\end{equation}
satisfies \eqref{OEPGeneral}, where $\partial \Omega = \bigcup _{i=1}^k \Gamma _i$, $N(\zeta _i)=\Gamma _i$, and $\abs{b_i}$ is the radius of the sphere in which $\zeta _i \in \pi _0 ( \partial \Sigma)$ is contained.
\end{proposition}

From now on, up to a dilation in $\r ^3$, $\Sigma \subset \mathbb{B}^3$ will denote an embedded minimal annulus with free boundaries, $\partial \Sigma= \zeta_1 \cup \zeta_2$, such that $\zeta_1 \subset \mathbb{S}^2$ and $\zeta_2 \subset \mathbb{S}^2 (\tilde r)$ for some $0<\tilde r \leq 1$. We will always assume that the boundaries of $\Sigma $ intersect the spheres \textit{from the inside}, i.e., let $\nu: \partial \Sigma \to \mathbb{S}^2$ the exterior conormal to $\Sigma$ then, $\pscalar{p}{\nu (p)} >0$ for all $p \in \partial  \Sigma $. 

Now, we are ready to prove Theorem B. We rewrite it here using the notation introduced so far.
\begin{quote}\label{uniquenessFreeBoundaries}
{\bf Theorem B:} {\it Let $\Sigma \subset \mathbb{B}^3$ be an embedded minimal annulus with free boundaries, $\partial \Sigma= \zeta_1 \cup \zeta_2$, such that $\zeta_1 \subset \mathbb{S}^2$ and $\zeta_2 \subset \mathbb{S}^2 (\tilde r)$ for some $0< \tilde r \leq 1$. Suppose that its support function has infinitely many critical points. Then, there exists $R \in [0,1)$ such that $\Sigma = C_R$, up to a rotation around the origin, where $C_R$ is one of the model catenoids defined in  \eqref{modelCatenoids}.}
\end{quote}
\begin{proof}[Proof of Theorem B]
First, since $\Sigma$ is minimal, the Hopf differential is holomorphic and so the interior zeros are isolated and of negative index. Moreover, since $\textup{cl}(\Sigma)$ intersects some sphere at a constant angle, Joachimstahl's theorem \cite[p. 152]{dC} implies that each boundary component is a line of curvature, and so the zeros at the boundary are also isolated and of negative index. Since the zeros of the Hopf differential correspond to umbilic points we have that the index at any umbilical point is negative. Using that $\Sigma$ is a topological annulus, the Poincar\'e-Hopf index theorem asserts that $\Sigma$ is either totally umbilical, which is not possible, or there are no umbilic points (cf. \cite{Ho,JNit} for details).

Next, since there are no umbilic points, the Gaussian curvature $K$ is strictly negative in ${\rm cl}(\Sigma)$ and the boundary curves $\zeta _1 \subset \s ^2 $ and $\zeta _2 \subset \s ^2 (\tilde r)$ are strictly convex spherical curves. This also implies that the Gauss map $N : \textup{cl}(\Sigma) \to \s ^2 $ is a local diffeormorphism. Moreover, $N:\zeta_i\to N(\zeta_i) = \Gamma _i$, $i=1,2$, is one to one. 

Consider the projection map $\varrho : \s ^2 (\tilde r) \to \s ^2 $ given by $\varrho (p) = p /|p|$. We next show:

\begin{quote}
{\bf Claim A:} {\it Let $\Sigma$ be an embedded minimal annulus with free boundaries. Then $\mathbb{S}^2 \setminus (\zeta_1 \cup \varrho (\zeta_2))= A \cup B_1 \cup \varrho(B_2)$, where $A$ is a topological annulus and $B_1 \subset \s ^2$ and $\varrho(B_2)$, $B_2 \subset \s ^2 (\tilde r)$, are disjoint topological disks with boundaries $\zeta_1$ and $\varrho (\zeta_2)$ respectively. In addition, $B_1$ and $\varrho(B_2) $ are contained in two different closed hemispheres of $\mathbb{S}^2$.}
\end{quote}
\begin{proof}
Let $F_i = \int_{\zeta_i} \nu$ be the flux along $\zeta_i$ (for $i=1,2$). Then, since $\Sigma$ is a minimal annulus, we have $F_1=-F_2$. But  $\Sigma$ has free boundaries, then it follows that $\nu (p) = p/|p|$ for each $p \in \zeta_i$, $i=1,2$, so we obtain that $\zeta_1$ and $\varrho(\zeta_2)$ are convex curves contained in $\mathbb{S}^2$ with opposite center mass.
	
On the other hand, since the convex hull of $\zeta_1$ (resp. $\varrho(\zeta _2)$), ${\rm conv}(\zeta_1)$ (resp. ${\rm conv}(\varrho(\zeta _2))$), is given by the intersection (see \cite{Sc}) of all closed halfspaces containing $\zeta_1$ (resp. $\varrho(\zeta _2)$), if $H$ is a closed halfspace of equation $\langle q,v_0\rangle\leq r_0$, for $v_0\in\r^3\setminus\{0\},r_0\in\r$, with $\zeta_1\subset H$ (resp. $\varrho(\zeta _2) \subset H$) then it is clear that $F_1\in H$ (resp. $F_2 \in H$). Moreover, $F_i\in\partial H$ if and only if $\zeta_i\subset\partial H$. Therefore, $F_1$ (resp. $F_2$) is a non zero vector with $F_1\in {\rm conv}(\zeta_1)\setminus\zeta_1$ (resp. $F_2\in {\rm conv}(\varrho(\zeta_2))\setminus \varrho(\zeta_2)$). This fact proves that both $\zeta_1$ and $\varrho(\zeta_2) $ cannot be contained in a common {\it closed} hemisphere of $\s^2$, otherwise $F_1+F_2\neq 0$. Thus the result is proven.
\end{proof}

Let $B_i$  be the convex inner domain inside $\mathbb{S}^2 (r_i)$, $r_1 = 1$ and $r_2 = \tilde r$, bounded by the boundary component $\zeta_i$, $i \in \{1,2\}$. Consider the closed cones of $\r^3$ given by
	\begin{equation}\label{conosss}
		{\cal C}_i=\{\lambda p:\ \lambda\geq0,\ p\in \zeta_i \cup B_i\},\quad i \in \{1,2\}.
	\end{equation}

From now on we fix the outward orientation $N$, that is, $N$ points towards $\s ^2 (r_i) \setminus \textup{cl}(B_i)$ along each boundary component $\zeta _i$. For every support plane $P$ to ${\cal C}_i$ consider its exterior unit normal $v_P$. The set ${\cal N}_i$ of exterior unit normals $v_P$ to ${\cal C}_i$ is a convex set of the sphere (see \cite{Sc}). Moreover, since $\Sigma$ meets $\s^2$ and $\mathbb{S}^2 (\tilde r)$ at constant angles, the boundary of ${\cal N}_i$ is determined by the curve $N(\zeta_i)$.
	
\begin{quote}
{\bf Claim B:} {\it$\mathcal N_1 \cap \mathcal N _2 = \emptyset$.}
\end{quote}
\begin{proof}[Proof of Claim B]
First, by contradiction, assume that $N(\zeta_1)\cap N(\zeta_2) \neq \emptyset$. Then, there would exist $N_0\in N(\zeta_1)\cap N(\zeta_2)$, and so the plane $P$ with equation $\langle q,N_0\rangle=0$ satisfies that $\zeta_1 \cup \zeta_2$ is contained in a closed halfspace determined by $P$. But, this contradicts that $\zeta_1$ and $\varrho(\zeta_2)$ cannot be contained in a common {\it closed} hemisphere of $\s^2$ by Claim A. 
Second, by contradiction again, assume $\mathcal N _1 \subset \mathcal N _2$. An argument as above shows that $\zeta _1$ and $\varrho(\zeta_2)$ are contained in the same hemisphere. Which is a contradiction.
\end{proof}	

Therefore, by Claim B, $\s^2\setminus(N(\zeta_1)\cup N(\zeta_2))$ has three connected components: two open disks whose closure are ${\cal N}_1 $ and ${\cal N}_2$ and an open annulus $\Omega$ whose boundary curves are $N(\zeta_1) \cup N(\zeta_2)$. Hence, as the degree of a differential map between compact manifolds is the same in all regular values (see \cite[Chapter 5, Lemma 1.4]{Hir}) an elementary topological argument shows that $N$ is a global diffeomorphism from $\textup{cl}(\Sigma)$ onto the closure of the annulus $\Omega := N (\Sigma) \subset \s^2$. Thus, we can pushforward the support function $u (p):= \meta{p}{N(p)}$, $p \in \textup{cl}(\Sigma)$, to $\textup{cl}(\Omega) := N(\textup{cl}(\Sigma))$, $\xi$ defined in \eqref{supportFunction}, and $\xi$ satisfies \eqref{OEPGeneral} and has infinitely many critical points. 

As a consequence of the injectivity of the Gauss map, we will deduce that the support function of $\Sigma$ must have a constant sign.
	
\begin{quote}
{\bf Claim C:} {\it Let $\Sigma $ be an embedded minimal annulus with free boundaries. Then, 
$$ \Sigma \cap  {\cal C}_i  = \emptyset , \quad i \in \{1,2\}, $$where ${\cal C}_i$ are the closed convex cones defined by (\ref{conosss}).}
\end{quote}
\begin{remark} 
This result can be easily thought of if we replace the cones ${\cal C}_i$ with a convex body whose boundary is a smooth surface and agrees with the cone except in a small neighborhood of the vertex. And also replace support planes with tangent planes.
\end{remark}
\begin{proof}[Proof of Claim C]
Let $v_i\in B_i\subset{\cal C}_i$ a fixed vector contained in the interior of ${\cal C}_i$. We first observe that the cone $\mu v_i+{\cal C}_i$ given as the translation of ${\cal C}_i$ using the vector $\mu v_i$ is included in the interior of ${\cal C}_i$ for each $\mu>0$. 

Then, by Claim A, $(\mu v_i+{\cal C}_i)\cap \textup{cl}(\Sigma) =\emptyset$ for $\mu $ large enough. Now, decrease $\mu$ until $(\mu v_i+{\cal C}_i)\cap \overline\Sigma\neq\emptyset$ for the first time. Denote by $\mu_0$ that $ \mu$.
	
If $\mu_0>0$ then $\mu_0 v_i+{\cal C}_i$ intersects $\textup{cl}(\Sigma)$ at an interior point $p_0\in\Sigma$. So, the tangent plane of $\Sigma$ at $p_0$ is a support plane of $\mu_0 v_i+{\cal C}_i$ (or equivalently, a support plane of ${\cal C}_i$), with the same interior unit normal $N(p_0)$. This is a contradiction because the image of the interior unit normal of $\Sigma$ and ${\cal C}_i$ are two disjoint sets of the sphere by Claim B.
	
If $\mu_0=0$ then $\zeta_i\subset{\cal C}_i\cap \textup{cl}(\Sigma)$. But it is impossible the existence of another point $p_0\in\Sigma$ at ${\cal C}_i\cap \textup{cl}(\Sigma)$ using the same reasoning as in the case $\mu_0>0$.
\end{proof}

It is important to recall that we are considering the outward orientation. Then, Claim C implies that the support function $u (p) = \meta{p}{N(p)} $ is positive at some point on $\Sigma$.

\begin{quote}
{\bf Claim D:} {\it Let $ \Sigma $ be an embedded minimal annulus with free boundaries. Then, its support function $u (p)$ is strictly positive in $\Sigma$ and vanishes along $\partial \Sigma$. Here, $N$ is the outward orientation.}
\end{quote}
\begin{proof}[Proof of Claim D]
First, we will show that $ \Sigma \cap L({\bf o},p) = \set{p}$ for all $p \in \Sigma$, where $L({\bf o},p)$ is the half-line starting at the origin ${\bf o}$ passing through $p \in \Sigma$. 
	
Assume this was not the case. Consider the dilation $\Sigma _\lambda :=\lambda \Sigma$, $\lambda \geq 1$. Observe that $\partial \Sigma _\lambda \subset \mathcal C _1 \cup \mathcal C _2$, where $\mathcal{C}_i$, $i=1,2$, are the closed cones defined in \eqref{conosss}. Then, for $\lambda >1 $ big enough, $\Sigma \cap \Sigma _\lambda = \emptyset$. If there were a point $q \in \Sigma $ so that $ \abs{\Sigma \cap L({\bf o},p)} \geq 2$, then there exists $\bar \lambda >1$ so that $\Sigma \cap \Sigma _{\bar \lambda } \neq \emptyset$ and $\Sigma \cap \Sigma _{ \lambda } =\emptyset$ for all $\bar \lambda < \lambda$, observe that the intersection points must be at the interior of the surfaces. This implies that the tangent planes at the intersection points $\Sigma \cap \Sigma _{\bar \lambda }$ must be equal and, also, the normals must be equal. But this contradicts that the Gauss map is a global diffeomorphism.
	
Therefore, $u(p) = \meta{p}{N(p)} \geq 0$ (with the outward orientation) on $\textup{cl}(\Sigma)$. Since $u = 0$ along $\partial \Sigma$ and $u$ is a Jacobi function, it is standard that either $u\equiv 0$, which is impossible, or $u >0 $ on the interior. This finishes the proof.
\end{proof}
\begin{remark}
When $\Sigma \subset \b ^3$ is an embedded free boundary minimal surface, Kusner and McGrath in \cite[Corollary 4.2]{Kus} proved that $\Sigma$ is a radial graph. However, their proof depends strongly on the two-piece property for embedded free boundary minimal surfaces in $\b^3$ (see \cite{Lima}), a property that is not clear to hold in our case.
\end{remark}

Observe that Claim D implies that $\textup{cl}(\Sigma)$ is a radial graph. Finally, we need to prove:

\begin{quote}
{\bf Claim E:} {\it If $ \# {\rm Crit}(u) = + \infty$, then $ \tilde \gamma = {\rm Max}(u) = {\rm Crit}(u)$ is a closed simple curve.}
\end{quote}
\begin{proof}[Proof of Claim E]
If $\# {\rm Crit} (u) = +\infty$, by analyticity, we already know that there exists, at least, one simple curve $\tilde \gamma$ contained in ${\rm Crit}(u)$; we will see that this is the only one. Consider $\varrho : \textup{cl}(\Sigma) \to \s ^2 $ the radial projection $\varrho(p) = p / |p|$ and $ N : \textup{cl}(\Sigma) \to \s ^2 $ the outward orientation.

Is it easy to observe that $p \in {\rm Crit}(u)$ if, and only if, $\varrho (p) = N (p)$ and hence $ \gamma = \varrho (\tilde \gamma) =  N (\tilde \gamma)$. Without loss of generality, up to a rotation, we can assume that the flux along the boundary components of $ \Sigma$ is vertical. Thus, since $\textup{cl}(\Sigma)$ is a radial graph, $\gamma $ divides $\s ^2$ into two components $S^+ \cup S^- = \s ^2 \setminus \gamma$ where $S^+$ contains the north pole and $S^-$ contains the south pole. 

Moreover, $ \gamma \subset \varrho(\textup{cl}(\Sigma)) \cap  N (\textup{cl}(\Sigma))$ divides both set into two components, that is, 
$$ \mathcal F ^\pm  = \varrho(\textup{cl}(\Sigma)) \cap S ^\pm \text{ and } \Omega ^\pm  = \ N(\textup{cl}(\Sigma)) \cap S ^\pm  $$

Since both, $\varrho$ and $N$, are one-to-one, they map a component $\textup{cl}(\Sigma) \setminus \tilde \gamma$ into a component of their image. Let $\Sigma ^+ $ be the connected component of $\Sigma \setminus \tilde \gamma$ whose boundary component $\zeta ^+ = \partial \Sigma^+ \setminus  \tilde \gamma$ has flux in the positive vertical axis. Thus, by the free boundary condition $N (\Sigma ^+) = \Omega ^-$ and $\varrho(\Sigma ^+) = \mathcal F ^+$, which implies that $N (\Sigma ^+) \cap F(\Sigma ^+) = \emptyset$. The same happens on the other component $\Sigma ^-$ of $\Sigma \setminus \tilde \gamma$ and boundary component $\zeta ^- = \partial \Sigma^- \setminus \tilde\gamma$. Therefore 
$$ N (\Sigma \setminus \tilde \gamma ) \cap \varrho(\Sigma \setminus \tilde \gamma) = \emptyset , $$which means that the only critical points of $u $ are those of $\tilde \gamma$. 

So far, we have proven that $\tilde \gamma = {\rm Crit}(u) = {\rm Crit} (d)$, where $d(p) = |p|^2$ is the square distance function. Since $\Sigma$ is minimal, the only possibility is that $\tilde \gamma$ is a curve of absolute minimums of $d$, that is, there exists $\tilde R \in (0,1)$ so that $\tilde \gamma \subset \s^2 (\tilde R)$, $\Sigma$ is tangential to $\s ^2 (\tilde R)$ along $\tilde \gamma$, and $\tilde \gamma$ is line of curvature of $\Sigma$ by Joachimstahl's theorem \cite[p. 152]{dC}. This proves Claim E.
\end{proof}

Then, we conclude that $(\Omega:=N(\Sigma), \xi := u \circ N^{-1} )$ is a solution to \eqref{OEP} satisfying the conditions of Theorem A since $\gamma = N (\tilde \gamma) = {\rm Max}(\xi)$ by Claim E. Hence $(\xi , \Omega )\equiv (\xi_R, \Omega _R)$ for some $R \in [0,1)$, this finishes the proof of Theorem B.
\end{proof}

Finally, Corollary C follows as an immediate consequence of Theorem B, taking into account that $C_0$ is the only free boundary catenoid of the family in $\b ^3$.

\section*{Acknowledgments}
\mbox{}

The first author is partially supported by Spanish MIC Grant PID2020-118137GB-I00 and MIC-NextGenerationEU Grant 30.RP.23.00.04 CONSOLIDACION2022.

The second author is partially supported by the {\it Maria de Maeztu} Excellence Unit IMAG, reference CEX2020-001105-M, funded by MCINN/AEI/10.13039/ 501100011033/CEX2020-001105-M.

\end{document}